\documentclass{amsproc}

\usepackage{amsmath}
\usepackage{amssymb}
\usepackage{amsfonts}
\usepackage{tikz}
\usetikzlibrary{decorations.markings}
\usepackage{tikz-cd}
\usepackage{slashed}
\usepackage{braket}
\usepackage{extarrows}
\usepackage[margin=1.3in]{geometry}
\usepackage{textcomp}
\usepackage{mathtools}
\usepackage{hyperref}
\usepackage{mathrsfs}
\usepackage{tabularx}
\usepackage[title]{appendix}
\usetikzlibrary{backgrounds}
\allowdisplaybreaks

\newcommand{\Tr}{\mathrm{Tr}}
\newcommand{\Gal}{\mathrm{Gal}}
\newcommand{\GW}{\mathrm{GW}}

\newcommand{\bP}{\mathbb P}

\newcommand{\bQ}{\mathbb Q}
\newcommand{\bR}{\mathbb R}
\newcommand{\bZ}{\mathbb Z}
\newcommand{\bF}{\mathbb F}

\newcommand{\cN}{\mathcal N}

\newcommand{\Neck}{\mathrm{Neck}}
\newcommand{\Emb}{\mathrm{Emb}}
\newcommand{\Orb}{\mathrm{Orb}}

\newcommand{\SN}{\Orb(C_n,\Neck(n,j))}
\newcommand{\SoN}{\Orb_{odd}(C_n,\Neck(n,j))}
\newcommand{\SeN}{\Orb_{even}(C_n,\Neck(n,j))}
\newcommand{\SNhalf}{\Orb(C_{\frac n2},\Neck(\frac n2,\frac j2))}
\newcommand{\SoNhalf}{\Orb_{odd}(C_{\frac n2},\Neck(\frac n2,\frac j2))}
\newcommand{\SeNhalf}{\Orb_{even}(C_{\frac n2},\Neck(\frac n2,\frac j2))}
\newcommand{\SNt}{\Orb(C_{2j}, \Neck(2j,j)^\tau)}

\newcommand{\SeNt}{\Orb_{even}(C_{2j}, \Neck(2j,j)^\tau)}
\newcommand{\SNjj}{\Orb(C_{2j}, \Neck(2j,j))}

\newtheorem{theorem}{Theorem}[section]
\newtheorem{lemma}[theorem]{Lemma}
\newtheorem{proposition}[theorem]{Proposition}
\newtheorem{corollary}[theorem]{Corollary}

\theoremstyle{definition}
\newtheorem{definition}[theorem]{Definition}
\newtheorem{example}[theorem]{Example}

\theoremstyle{remark}
\newtheorem{remark}[theorem]{Remark}

\numberwithin{equation}{section}

\begin{document}

\title{Quadratically enriched binomial coefficients over a finite field}


\author{Chongyao Chen}
\address{Duke University, Durham, NC, USA}
\curraddr{}
\email{chongyao.chen@duke.edu}
\thanks{CC was partially supported by National Science Foundation Awards DMS-2304981 and DMS-2405191}

\author{Kirsten Wickelgren}
\address{Duke University, Durham, NC, USA}
\curraddr{}
\email{kirsten.wickelgren@duke.edu}
\thanks{KW was partially supported by National Science Foundation Awards DMS-2103838 and DMS-2405191}


\date{December 2024}

\begin{abstract}
We compute an analogue of Pascal's triangle enriched in bilinear forms over a finite field. This gives an arithmetically meaningful count of the ways to choose $j$ ring homomorphisms into an algebraic closure from an \'etale extension of degree $n$. We also compute a quadratic twist. These (twisted) enriched binomial coefficients are defined in joint work of Brugall\'e and the second-named author, building on work of Serre. Such binomial coefficients support curve counting results over non-algebraically closed fields, using $\mathbb{A}^1$-homotopy theory.

\end{abstract}

\maketitle

\section{Introduction}

We consider combinatorics enriched in bilinear forms, in the sense that an integer $n$ is replaced by the class of a symmetric, non-degenerate bilinear form on a vector space of dimension $n$. The resulting binomial coefficients arose in \cite{Brugalle-WickelgrenABQ} in the context of curve counting over non-algebraically closed fields: to count curves on surfaces, one is led to certain degeneration formulas in which curves of lower degrees are glued together. To perform such gluings, one chooses closed points. However, over non-algebraically closed fields, the points have different residue fields. To obtain a count retaining arithmetic information, it is effective to replace integer valued counts with ones enriched in bilinear forms. The counts now take values in the Grothendieck--Witt group of the base field, defined to be the group completion of isomorphism classes of symmetric, non-degenerate bilinear forms. See for example \cite{cubicsurface} \cite{Levine-EC} \cite{Wendt-oriented_Schubert_calculus} \cite{PMPR-tropical_GW_invts} \cite{McKean-Bezout} \cite{Cotterill-Darago-Han}. In \cite{Brugalle-WickelgrenABQ}, Erwan Brugall\'e and the second named author obtained a wall-crossing formula for $\mathbb{A}^1$ Gromov--Witten invariants using combinatorial identities enriched in bilinear forms. Here we systematically compute the analogue of Pascal's triangle over a finite field of odd characteristic and a quadratic twist, arising in the enumerative context of \cite{Brugalle-WickelgrenABQ},\cite{degree}.

 Let $k =  \bF_q$ be a finite field with odd characteristic. Let $\GW(k)$ denote the Grothendieck--Witt group of $k$ and let $u \in \GW(\bF_q)$ denote the class of the bilinear form $\bF_q \times \bF_q \to \bF_q$ sending $(x,y)$ to $\mu x y$, where $\mu$ is a non-square in $\bF_q^*$. There is a ring structure on $\GW(k)$ induced from the tensor product of forms, and the unit, denoted by $1$, is represented by the bilinear form $\bF_q \times \bF_q \to \bF_q$ sending $(x,y)$ to $xy$. For an \'etale $k$-algebra $L$, the paper \cite{Brugalle-WickelgrenABQ} defines $\binom{L/k}{j} $ in $\GW(k)$ to be the trace form of the \'etale $k$-algebra associated via the Galois correspondence to the set of subsets of size $j$ of the set of $k$-maps of $L$ into $\overline{k}$. Here, $\overline{k}$ denotes the algebraic closure of $k$. 
 
 For a quadratic extension $Q$ of $k$ and $[L:k] = 2j$, this set has a twisted action, where $\mathrm{Gal}(\overline{Q}/k)$ acts by swapping two complementary subsets. This defines a quadratically twisted binomial coefficient $\binom{L[Q]/k}{j}$ in $\GW(k)$. We give these definitions in detail in Section~\ref{Section:Quadratically_enriched_binomial_coefficients}. 
 
In this paper, we show the following closed formula for the quadratically enriched binomial coefficients over $\bF_q$.
\begin{theorem}\label{thm:closedformula-1}
    Let $q$ be an odd prime power and let $j$ be a non-negative integer. Let $L/k$ be the finite extension of $\bF_q$ of degree $n$. Then 
    \[
    \binom{L/k}{j} = \binom{n}{j} - (1 -u)\cdot \binom{\frac{n-2}{2}}{\frac{j-1}{2}} \in \GW(\bF_q),
    \]
    where $u$ is the non-square class in $\GW(k)$ and our convention is that $\binom{a}{b}:=0$, if either $a$ or $b$ is not in $\bZ$.
\end{theorem}

Theorem~\ref{thm:closedformula-1} gives an enrichment of Pascal's triangle in $\GW(\bF_q)$. The first few rows $n=0,1,\ldots 8$ of the (untwisted, $\GW(\bF_q)$-enriched) Pascal's triangle ${\bF_{q^n}/\bF_q \choose j }$ are the following: 

\[
\begin{array}{ccccccccccccccccccc}
     &  &  &  &  &  &  &  &  &  & 1 &  &  &  &  &  &  &  &  \\
     &  &  &  &  &  &  &  &  & 1 &  & 1 &  &  &  &  &  &  &  \\
     &  &  &  &  &  &  &  & 1 &  & 1+u &  & 1 &  &  &  &  &  &  \\
      &  &  &  &  &  &  & 1 &  & 3 &  & 3 &  & 1 &  &  &  &  &  \\
     &  &  &  &  &  & 1 &  & 3+u &  & 6 &  & 3+u &  & 1 &  &  &  &  \\
     &  &  &  &  & 1 &  & 5 &  & 10 &  & 10 &  & 5 &  & 1 &  &  &  \\
     &  &  &  & 1 &  & 5+u &  & 15 &  & 20 &  & 15 &  & 5+u &  & 1 &  &  \\
     &  &  & 1 &  & 7 &  & 21 &  & 35 &  & 35 &  & 21 &  & 7 &  & 1 &  \\
     &  & 1 &  & 7+u &  & 28 &  & 55+u &  & 70 &  & 55+u &  & 28 &  & 7+u &  & 1 \\
\end{array}.
\]

We also compute the twisted quadratically enriched binomial coefficients over a finite field:

\begin{theorem}\label{thm:closedformula-2}
    Let $q$ be an odd prime power and let $j$ be a non-negative integer. Let $L/k$ be the finite field extension of $\bF_q$ of degree $2j$, and let $Q/k$ the degree $2$ field extension of $\bF_q$. Then
    \[
    \binom{L[Q]/k}{j}= \frac12\binom{2j}{j}\cdot(1+u) \in \GW(\bF_q).
    \]
\end{theorem}

\begin{example}\label{ex:twisted-enriched-bin-finite-field}
    The first few terms of the sequence ${L[Q]/k \choose j}$ are \[2,5+u,20,69+u,252,924,3432,12869+u,\dots.\]
\end{example}

\subsection{Summary of the proof}
The idea of the proofs is first to rephrase the problems into purely combinatorial forms. In Section \ref{subsection:necklace_interpretation} and \ref{sec:twistenecint}, we show that the proofs of these two theorems reduce to calculations of the parities of the number of orbits of necklaces of even cardinality under different cyclic group actions.

\subsubsection{Untwisted case} The proof of Theorem~\ref{thm:closedformula-1} in the case when $ n $ is odd is very straightforward, as shown in Proposition \ref{prop:case1}.  
In Section \ref{subsection:nevenjodd}, we prove Theorem \ref{thm:closedformula-1} when $ j $ is odd. The proof involves a direct calculation using M\"obius inversion and Lucas's theorem. The remaining task is to show that when both $ n $ and $ j $ are even, the enumeration is always even, which is established combinatorially.

In Section \ref{subsection:comb1}, we use a $ C_2 $-action, referred to as flipping, on the set of cyclic orbits of necklaces. Since we are concerned only with the parity of the enumeration, it suffices to count the $ C_2 $-fixed points. Next, in Section \ref{subsection:comb2}, we introduce the concept of symmetry axes for the orbits of necklaces, which allows us to rewrite the set of $ C_2 $-fixed (symmetric) orbits as a non-disjoint union of two subsets: the sets of symmetric orbits with a symmetry axis passing through two beads (we call this a symmetry axis of type 1) and those with a symmetry axis passing between two pairs of beads (type 2). The enumeration methods for these two subsets are different.

In Section \ref{subsection:comb2}, we also decompose each cyclic orbit of necklaces into two smaller necklace orbits. This decomposition exhibits special properties when the orbit is symmetric, as different types of symmetry axes yield different properties. This decomposition induces a map whose codomain is the symmetric product of the sets of two smaller cyclic orbits of necklaces. The enumeration is carried out by summing the fibers of this map over the codomain. At the end of Section \ref{subsection:comb2}, we provide the enumeration of symmetric orbits with a type 1 symmetry axis when both $ n $ and $ j $ are even.

Using this approach, in Section \ref{subsection:case2}, we enumerate the symmetric orbits with a type 2 symmetry axis for the case where $ n \equiv 2 \mod 4 $ and $ j $ is even. Combining this result with the results in Section \ref{subsection:comb2}, we prove Theorem \ref{thm:closedformula-1} for the case where $ n \equiv 2 \mod 4 $ and $ j $ is even by calculating the parity of the enumeration using Kummer's theorem.

In Section \ref{subsection:case3}, we consider the case where $ n \equiv 0 \mod 4 $. For a symmetric orbit with a type 2 symmetry axis, we show how to reduce it to a symmetric orbit with a type 1 symmetry axis. Since $ n - 2 \equiv 0 \mod 4 $, we can now utilize the results from Section \ref{subsection:comb2} and Section \ref{subsection:case2} to conclude the proof of Theorem \ref{thm:closedformula-1} for the case where $ n \equiv 0 \mod 4 $ and $ j $ is even. This completes the proof of the untwisted case.

\subsubsection{Twisted case} We prove Theorem \ref{thm:closedformula-2} by reducing to an untwisted enumeration that we have studied in detail previously. In Section \ref{subsection:reduction}, we first provide a description of the twisted orbits in terms of untwisted orbits. Next, we construct a $C_2$-action, called swapping, on the set of twisted orbits and reduce the problem to the enumeration of the swapping fixed points. Then, by studying the swapping of fixed twisted orbits, we reduce the problem to an enumeration of untwisted orbits under certain conditions.

In Section \ref{subsection:relatetopart}, we discuss another way to encode the information of an untwisted cyclic orbit of necklaces. This is the marked cyclic equivalence class of partitions. The condition for the untwisted orbits we need to enumerate can be rephrased in terms of the properties of the partitions, which allows us to calculate the parity of the enumeration using partition theory. This concludes the proof of Theorem \ref{thm:closedformula-2}.

More explicit forms of the untwisted and twisted binomial coefficients over a finite field are presented in Section \ref{subsection:expval1} and Section \ref{subsection:expval}, respectively.

\subsection{Acknowledgements} We heartily thank De'Asia Brodie and Zoe Valentine for collaboration on the cases $j=1,2,3$ of Theorem~\ref{thm:closedformula-1}, and Erwan Brugall\'e and Rena Chu for useful discussions. KW received support from NSF DMS-2103838 and NSF DMS-2405191. CC is partially supported by DMS-2304981 and DMS-2405191.

\section{Background}

The M\"obius inversion formula \cite[16.4]{HardyWright} states that for two functions $f,g:\mathbb{Z}_{> 0} \to \mathbb{C}$ such that $g(n) = \sum_{d \vert n} f(n)$ we have that 
\[
f(n) = \sum_{d \vert n} \mu(d) g(\frac{n}{d})
\] where $\mu$ is the M\"obius function
\begin{equation}\label{eq:mu_def_Mobius_inversion}
\mu(j) =
\begin{cases} 
1 & \text{if } j = 1, \\
0 & \text{if } j \text{ has a squared prime factor}, \\
(-1)^v & \text{if } j \text{ is a product of } v \text{ distinct prime numbers}.
\end{cases}
\end{equation}

\subsection{Lucas and Kummer's theorems}
We recall the classical results of Lucas and Kummer.
For a prime $p$, let $\nu_p:\bZ \to \bZ_{\geq0} $ denote the $p$-adic valuation map.
Lucas' classical theorem calculates the mod $p$ residue class for binomial coefficients.
\begin{theorem}\label{thm:lucastheorem}[Lucas's theorem \cite{lucastheorem}]
For $x,y\in\bZ_{\geq 0}$, and $p$ a prime
\[
\binom{x}{y} \equiv \prod_{i}\binom{x_i}{y_i}\mod p
\]    
where $x_i,y_i$ are the coefficients of the $p$-adic expansions of $x$ and $y$, $x = \sum_{i}x_ip^i$, $y = \sum_i y_i p^i$.
\end{theorem}

Define the $p$-adic carrier function for an integer $x$ as $S_p(x) = \sum_i{x_i}$, where $x = \sum_{i}x_ip^i$ is the $p$-adic expansion. Kummer's classical theorem calculates the $p$-adic valuation of binomial coefficients.
\begin{theorem}[Kummer's theorem \cite{Kummertheorem}]
    The $p$-adic valuation of $\binom{n}{m}$ is
    \[
    \nu_p\binom{n}{m} = \frac{S_p(m)+S_p(n-m)-S_p(n)}{p-1}.
    \]
\end{theorem}

Some easy but useful corollaries that are relevant to our purpose are the following.
\begin{corollary}\label{cor:corKummer1}
For $n,j\in\bZ$, we have $\nu_2\binom{n}{j} = \nu_2\binom{2n}{2j} = \nu_2\binom{2n+1}{2j+1}$.  
\end{corollary}
\begin{corollary}\label{cor:corKummer2}
For $n,j\in2\bZ$ we have 
\[
\nu_2\binom{n}{j} = \nu_2\binom{n+1}{j}.
\]
\end{corollary}
\begin{corollary}\label{cor:corKummer3}
    For $j\in\bZ_{>0}$, we have $\nu_2\binom{2j}{j}\geq1$, and the equality holds if and only if $j= 2^m$, $m\in\bZ_{>0}$.
\end{corollary}

\section{Quadratically enriched binomial coefficients over a field}\label{Section:Quadratically_enriched_binomial_coefficients}

The paper \cite{Brugalle-WickelgrenABQ} defined quadratically enriched (twisted) binomial coefficients over a base scheme. We recall the definition in the case where the base scheme is a field.

Let $k$ be a field, and fix a separable closure $k^s$ of $k$. For a finite separable extension $k \subseteq L$, let $\Emb_k(L,k^s)$ denote the set of ring homomorphisms of $L$ into $k^s$ over $k$ 
\[
\Emb_k(L,k^s) = \left\{ f : \begin{tikzcd}
L \arrow[rr, "f"] && k^s  \\
&k \arrow[lu, ""] \arrow[ru, ""] 
\end{tikzcd} \right\}.
\] 
Here it is not necessary to assume that $L$ is a field; the same definition will carry over verbatim. But note that in this case $f$ is not injective. We keep the notation $\Emb$ because when $L$ is a field, $\Emb_k(L,k^s)$ the set of embeddings of $L$ into the separable closure.

Recall that a finite \'etale $k$-algebra $k \to E$ has an associated trace map
\[
\Tr_{E/k}: E \to k
\] which takes an element $e$ to the trace of the matrix associated to multiplication by $e$, viewed as an endomorphism of the finite dimensional $k$-vector space $E$. This trace map determines an element of $\GW(k)$ denoted $\Tr_{E/k} \langle 1 \rangle$ and defined to be the class of the bilinear form
\[
E \times E \to E \to k
\] where the first map is multiplication on $E$ and the second is $\Tr_{E/k}$.

Recall that the Galois correspondence
\[
E \mapsto \Emb_k(E,k^s)
\] gives an equivalence of categories from finite \'etale $k$-algebras to finite sets equipped with a $\Gal(k^s/k)$-action. Under this correspondence, a field $F$ corresponds to a set with a transitive action. A set with a transitive action can be encoded by the stabilizer $H \leq \Gal(k^s/k)$ of any element, and $F$ is recovered as the fixed field $F \cong (k^s)^H$ of the stabilizer.

To enrich binomial coefficients, we introduce the following notation. Let $\Emb^j_k(L,k^s)$ denote the set of subsets of size $j$ of $\Emb_k(L,k^s)$. So $\Emb^j_k(L,k^s)$ is the subset of the power set of $\Emb_k(L,k^s)$ consisting of those subsets of size exactly $j$. This set is equipped with a $\Gal(k^s/k)$-action coming from the action on $\Emb_k(L,k^s)$. Let $E^j_k(L,k^s)$ denote the finite \'etale $k$-algebra  that is associated to $\Emb^j_k(L,k^s)$ under the Galois correspondence.

\begin{definition}\label{df:binomialcoef}\cite[Definition 4.1]{Brugalle-WickelgrenABQ}
For a finite \'etale $k$-algebra $L$, let $$\binom{ L/k}{j} \in \GW(k)$$ denote the class of the trace form of $E^j_k(L,k^s)$. 

Here, $E^j_k(L,k^s)$ corresponds to the finite set $\Emb^j_k(L,k^s)$ with $\Gal(k^s/k)$-action induced from the canonical action on $\Emb_k(L,k^s)$.
\end{definition}

For $j=2,3$, these forms appeared in \cite[30.12-30.14]{Garibaldi-Serre-Merkurjev}.

\begin{remark}\label{rem:L/kchoosej-computation}
    By definition ${ L/k \choose j} = \Tr_{E^j_k(L,k^s)/k} \langle 1 \rangle$. We have $\dim_k E^j_k(L,k^s) = |\Emb^j_k(L,k^s)| =\binom{n}{j}$. For a group $G$ acting on a set $S$, let $\Orb(G,S)$ denote the set of orbits. The set $\Emb_k^j(L,k^s)$ can be decomposed into Galois orbits $\Emb_k^j(L,k^s) = \amalg_{i\in I} \mathfrak O_i$  with $$I=\Orb(\Gal(k^s/k), \Emb_k^j(L,k^s))$$ being a finite set. Let $H_i\leq \Gal(k^s/k)$ be the stabilizer subgroup of $\mathfrak O_i$. Then by the Galois correspondence $E^j_k(L,k^s) = \prod_{i\in I}(k^s)^{H_i}$, and 
\begin{equation}\label{eq:L/kchoosejfromstab}
{ L/k \choose j} = \Tr_{E^j_k(L,k^s)/k}\braket{1} = \sum_{i\in I}\Tr_{(k^s)^{H_i}/k}\braket{1}.
\end{equation} Note that any element of $\Gal(k^s/k)$ which stabilizes every element of $\Emb_k(L,k^s)$  also stabilizes every element of   $\Emb_k^j(L,k^s)$. It follows that $(k^s)^{H_i}$ embeds as a subfield of $L$.
\end{remark}

\begin{example}
${ L/k \choose 1} = { L/k \choose [L:k]-1} = \Tr_{L/k} \braket{1}$. For example, suppose $d \in \mathbb{Q}^*$ is not a square. Then $ {\mathbb{Q}[\sqrt{d}]/\mathbb{Q}\choose 1} = \braket{2} + \braket{2d}$, where $\braket{d}$ denotes the class of the bilinear form $\mathbb{Q} \times \mathbb{Q} \to \mathbb{Q}$ given $(x,y) \mapsto dxy$. For more on trace forms, see \cite{Conner-Perlis}.
\end{example}

In \cite{Brugalle-WickelgrenABQ}, it was important to twist such quadratically enriched binomial coefficients by a degree $2$-field extension $k \subset Q$ as follows. There is a canonical isomorphism $\Gal(Q/k) \cong C_2$, where $C_2$ denotes the cyclic group of order $2$. Let 
$$q_Q: \Gal(k^s/k) \to \Gal(Q/k) \cong C_2$$ denote the corresponding quotient map.

For a set $S$ of size $2j$, the set of subsets of $S$ of size $j$ has an action of $C_2$ given by taking a subset to its complement. If a group $G$ acts on $S$, the set of subsets of size $j$ inherits an action of $G \times C_2$ because taking a set to its complement commutes with any automorphism of $S$.

\begin{definition}\label{df:twistedbincoef}\cite[Definition 4.3]{Brugalle-WickelgrenABQ}
Let $L$ be a finite \'etale $k$-algebra and $j$ such that $\dim_k L=2j$. Define $${ L[Q]/k \choose j} \in \GW(k)$$ to be the class of the trace form of the \'etale $k$-algebra corresponding to $\Emb^j_k(L,k^s)$ with the $\Gal(k^s/k)$-action given by the homomorphism $$\Gal(k^s/k)\stackrel{(1,q_Q)}{\to} \Gal(k^s/k) \times C_2$$ and the canonical action of $\Gal(k^s/k) \times C_2$ on $\Emb^j_k(L,k^s)$.
\end{definition} 

\begin{example}
Suppose $a,d \in \mathbb{Q}^*$ are relatively prime non-squares. Let $L=\mathbb{Q}[\sqrt{d}]$ and $Q=\mathbb{Q}[\sqrt{a}]$. Then $ {L[Q]/ \mathbb{Q}\choose 1} = \braket{2} + \braket{2da}$.
\end{example}

\section{Untwisted case - Proof of Theorem \ref{thm:closedformula-1}}

Let $k =  \bF_q$ be a finite field of odd characteristic. We have
\[
\GW(\bF_q) \cong \frac{\bZ[u]}{(u^2-1, 2-2u)} \cong \frac{\bZ\cdot 1\oplus\bZ\cdot u}{(2-2u)} \cong\bZ\times \bF_q^*/(\bF^*_q)^2,
\]
where the first isomorphism is an isomorphism of rings, the second and third only respect group structures, and the third isomorphism is the product of the rank and the discriminant. Note that $\bF_q^*/(\bF^*_q)^2\cong \bZ/2\bZ$. Here as above, $u$ denotes the class of the bilinear form $k \times k \to k$ given by $(x,y) \mapsto \mu xy$ where $\mu$ in $\bF_q$ is a non-square.  See for example \cite[Theorem 3.5, Corollary 3.6]{lam05}. In particular, the class of any element in $\GW(\bF_q)$ is determined by the rank and the discriminant. 

Let $L/k$ be a finite extension of degree $n$. Then $\Gal(L/k)\cong C_{n}$, where $C_n$ is the cyclic group of order $n$ that is generated by the Frobenius automorphism $\varphi:x\to x^q$. See for example \cite[Section 14.3]{dummit2004abstract}. It is a classical fact\footnote{ A proof can be found in Lemma 58 of the first ArXiv version of \cite{cubicsurface}.} that for $[L:k] = n$, the class of the trace form of $L$ over $k$ in $\GW(k)$ is given:
\begin{equation}\label{eqn:tracecompfield}
\Tr_{L/k}\braket{1} = \epsilon(n):=\left\{\begin{array}{cc}
   n-1 + u,  & n\equiv 0\mod 2  \\
    n, & n\equiv 1\mod2
\end{array}
\right.    
\end{equation}
Indeed, the rank of $\Tr_{L/k}\braket{1}$ is $n$, so it suffices to compute the discriminant of $\Tr_{L/k}\braket{1}$. The discriminant of the trace form is the discriminant of the minimal polynomial of a generator. Since the Fr\"obenius acts as an $n$ cycle on the roots of a minimal polynomial for a generator of $k \subseteq L$, we have that $\Tr_{L/k}\braket{1}$ has square discriminant if and only if an $n$ cycle has even sign, which occurs if and only if $n$ is odd.

As in Remark~\ref{rem:L/kchoosej-computation}, let $$I=\Orb(\Gal(k^s/k), \Emb_k^j(L,k^s))$$ denote the set of orbits of the action of $\Gal(k^s/k)$ on $\Emb_k^j(L,k^s)$, and let $$\Emb_k^j(L,k^s) = \amalg_{i\in I} \mathfrak O_i$$ denote the decomposition of $\Emb_k^j(L,k^s)$ into orbits. Letting $d_i = \vert \mathfrak O_i \vert$ denote the size of the $i$th orbit, we see that 
\[
{ L/k \choose j} = \sum_{i \in I} \epsilon(d_i)
\] by Equation~\eqref{eq:L/kchoosejfromstab}.

 For a group $G$ acting on a set $S$, let $\Orb_{even}(G,S) \subseteq \Orb(G,S)$ denote the subset consisting of those orbits with an even cardinality. Define $\Orb_{odd}(G,S) $ similarly.

The set $\Orb (\Gal(k^s/k), \Emb_k^j(L,k^s) )$ decomposes
\[
\Orb(\Gal(k^s/k),\Emb_k^j(L,k^s))=GO_{odd}(L/k,j)\amalg GO_{even}(L/k,j),\]
where we use the abbreviations
\begin{align*}
GO_{odd}(L/k,j):=&\Orb_{odd} (\Gal(k^s/k), \Emb_k^j(L,k^s) ),\\
GO_{even}(L/k,j):=&\Orb_{even} (\Gal(k^s/k), \Emb_k^j(L,k^s) ).
\end{align*}

Let $\Delta(n,j)$ denote the difference $\Delta(n,j):=\binom{L/k}{j}-\binom{n}{j}$. The following lemma follows immediately from \eqref{eqn:tracecompfield}.
\begin{lemma}\label{lem:deltaGoeven}
$\Delta(n,j) = (u-1)\cdot |GO_{even}(L/k,j)|$.   
\end{lemma}
Therefore, to determine the value of $\Delta(n,j)$, one only needs to determine the mod 2 residue class of $|GO_{even}(L/k,j)|$.

\begin{proposition}[Theorem \ref{thm:closedformula-1}: case for $n$ odd]\label{prop:case1}
  $\binom{L/k}{j} = \binom{n}{j}$ for $n$ odd.
\end{proposition}
\begin{proof}
    As in Remark~\ref{rem:L/kchoosej-computation}, $(k^s)^{H_i}$ is a subfield of $L$. Thus, we have that $d_i|n$. Thus, if $n$ is odd, $GO_{even}(L/k,j)=\emptyset$.
\end{proof}

\subsection{Necklace interpretation}\label{subsection:necklace_interpretation} Let $\Neck(n,j)$ denote the set of necklaces consisting of $j$ blue beads and $n-j$ red beads, with one bead designated as the bead at the top of the necklace. One could alternately view this set as the set of necklaces with $j$ blue beads and $n-j$ red beads, lying on the plane, beads equally spaced on the unit circle with the top bead at $(0,1)$. It will be useful to consider the natural action of the dihedral group $D_n = C_n \rtimes C_2 = \langle r, f : r^n = 1, f^2 =1, rfr = f \rangle$ on $\Neck(n,j)$ in which $r$ acts by rotating the necklace one bead counterclockwise and $f$ acts by flipping the necklace over the vertical line through the top bead.

As above, let $k \subseteq L$ be the degree $n$ field extension $\bF_q \subseteq \bF_{q^n}$. Fix an embedding $p \in \Emb_k(L,k^s)$. Then 
\[
\Emb_k(L,k^s) = \{p, \varphi p, \varphi^2 p\,\ldots, \varphi^{n-1} p\},
\] where $\varphi$ denotes the Fr\"obenius $\varphi \colon k^s \to k^s$, $\varphi (x) = x^q$ of $k$. It follows that if we identify $\Emb^j_k(L,k^s)$ with the set $\Neck(n,j)$, then the action of $\Gal(L/k)\cong C_n$ on $\Emb^j_k(L,k^s)$ is identified with the induced $C_n $-action  on $\Neck(n,j)$. The induced $C_n$-action comes from $C_n\to D_n$ which sends the generator to $r\in D_n$. In other words, we have constructed an isomorphism of $C_n$-sets
\[
\Emb^j_k(L,k^s) \cong \Neck(n,j).
\] In particular, $GO_{odd}(L/k,j)$ can be identified with the subset of orbits of $\Neck(n,j)$ under the rotation action of $C_n = \langle r : r^n =1 \rangle$ consisting of those orbits with odd cardinality. A similar statement holds for $GO_{even}(L/k,j)$ as well, giving a canonical bijection

\[
GO_{even}(L/k,j) \cong \Orb_{even}(C_n, \Neck(n,j)).
\] Moreover, Lemma~\ref{lem:deltaGoeven} says that 
\[
{L/k \choose j} = {n \choose j} + (u - 1) \vert \Orb_{even}(C_n, \Neck(n,j))\vert.
\]

\subsection{M\"obius inversion}
Let $N(n,j)$ denote the cardinality of the set of necklaces in $\Neck(n,j)$ whose stabilizer under the $C_n$-action by rotation is trivial
\[
N(n,j) := |\big\{l\in\Neck(n,j):\mathbf{Stab}_{C_n} (l)= \{1\}\big\}|.
\] By the necklace interpretation of Section~\ref{subsection:necklace_interpretation}, we have
$$N(n,j)=|\big\{l\in\Emb^j_k(L,k^s):\mathbf{Stab}_{\Gal(L/k)} (l) = \{1\}\big\}|.$$ Note that that the numbers $N(n,j)$ and $|\Emb^j_k(L,k^s)| = \binom{n}{j}$ only depend on $n$ and $j$. 

Let $I=\Orb(C_n, \Neck(n,j))$ index the orbits of the action of $C_n$ on $\Neck(n,j)$ and, for $i$ in $I$, let $d_i$ denote the cardinality of the corresponding orbit. Let $\rho:=\frac{j}{n}$ denote the fraction of beads which are blue. Given an orbit $i \in I$ of cardinality $d_i$, we can take $d_i$ adjacent beads in the necklace and form a new necklace with $d_i$ beads, $\rho d_i$ of which are blue, and which has a trivial stabilizer under the rotation action of $C_{d_i}$. This process can be run in reverse, creating a necklace with $n$ beads from one with $d$ beads for $d|n$. It follows that 
\[
|\{i\in I,d_i = d\}|= N(d,\rho d).
\]

By the necklace interpretation, $I$ is in canonical bijection with an indexing set for the orbits of the $\Gal(L/k)$ action on $\Emb^j_k(L,k^s)$. Combining with the above, we have
\[
\binom{n}{j} = |\Emb^j_k(L,k^s)|=\sum_{d|n}d~|\{i\in I,d_i = d\}| =  \sum_{d|n}d~N(d,\rho d),\quad \rho:=\frac{j}{n},
\]
under the convention $N(d,\rho d) = 0$ if $\rho d\notin \bZ$. 

We can now apply the M\"obius inversion formula, which yields

\[
|N(d,b)| = \frac{1}{d}\sum_{j|d}\mu(j)\binom{\frac{d}{j}}{\frac{b}{j}},
\]
where $\mu(j)$ is the M\"obius function recalled in Equation~\eqref {eq:mu_def_Mobius_inversion}.

By definition, $|GO_{even}(L/k,j)|$ is the total number of orbits of even size, thus
\begin{equation}\label{eqn:sizeGaloisorbit}
|GO_{even}(L/k,j)| = \sum_{d|n,2|d}\frac{1}{d}\sum_{j|d}\mu(j)\binom{\frac{d}{j}}{\frac{\rho d}{j}}.    
\end{equation}

\subsection{Proof of Theorem \ref{thm:closedformula-1}\label{subsection:nevenjodd}
when $\nu_2(n)\geq1$ and $\nu_2(j)=0$ }

Define a partial order $\prec$ on $\bQ_{\geq0}$ as follows. 

\begin{definition}\label{def:prec}
For $x,y\in \bQ_{\geq0}$ written as 2-adic numbers $x = \sum_{i} x_i\cdot 2^i$, $y = \sum_{i} y_i\cdot 2^i$, then $x\prec y$ if and only if $x,y\in\bZ_{\geq 0}$ and $x_i\leq y_i$ for all $i$. 
\end{definition}

\begin{proposition}\label{prop:case2}
For $\nu_2(n)\geq1$ and $\nu_2(j)=0$ we have
 \[
\Delta(n,j) = \left\{
\begin{array}{cc}
     -1+u, & \frac{j-1}{2}\prec \frac{n-2}{2}  \\
     0, & \mathrm{else}.
\end{array}
\right.
\]
\end{proposition}
\begin{proof}
As above, define $\rho = j/n$. From \eqref{eqn:sizeGaloisorbit}, we have
\begin{align*}
|GO_{even}(L/k,j)| &=  \sum_{d|n,2|d}\frac{1}{d}\sum_{l|d}\mu(l) \binom{\frac{d}{l}}{\frac{\rho d}{l}} .\\
&=  \sum_{d|n,2|d}\frac{1}{d}\sum_{l|d}\mu(l) \frac{1}{\rho} \binom{\frac{d}{l} -1}{\frac{\rho d}{l}-1} .\\
&= \frac{1}{j}\sum_{d|n,2|d}\frac{n}{d}\sum_{l|d} \mu(l)\binom{\frac{d}{l}-1}{\frac{\rho d}{l}-1}.
\end{align*}
By Lemma \ref{lem:deltaGoeven}, $\Delta(n,j)$ only depends on the mod 2 residue of $|GO_{even}(L/k,j)|$.
As $j$ is odd by our assumption,
\[
|GO_{even}(L/k,j)|\equiv \sum_{d|n, 2|d}\frac{n}{d}\sum_{l|d} \mu(l)\binom{\frac{d}{l}-1}{\frac{\rho d}{l}-1} \mod 2.
\]
Further, as $\mu(l)\equiv 1\mod2$ if and only if $l$ is square free,  we have
\begin{align}
\notag |GO_{even}(L/k,j)| & \equiv  \sum_{d|n,2|d}\frac{n}{d}\sum_{\substack{l | d \\ l \text{ square free}} }\binom{\frac{d}{l}-1}{\frac{\rho d}{l}-1} \\ \label{eq:GOevennevenjoddsumswpowers2} &\equiv\sum_{(2^{\nu_2(n)}d)|n}\sum_{\substack{l|(2^{\nu_2(n)}d)\\ l \text{ square free}} }\binom{\frac{2^{\nu_2(n)}d}{l}-1}{\frac{\rho 2^{\nu_2(n)}d}{l}-1}\mod 2.
\end{align}
Where in the second line we used the fact that only when $\frac{n}{d}$ is even, the summand has a nontrivial mod 2 contribution.
 By Corollary~\ref{cor:corKummer1}, we have mod $2$ equalities \[\binom{x-1}{y-1} \equiv \binom{2x-2}{2y-2}\equiv\binom{2x-1}{2y-1}\mod 2,\quad \forall x,y\in\bZ_{\geq0}.\] Thus, the mod 2 residue of each of the summands of \eqref{eq:GOevennevenjoddsumswpowers2} only depends on $\frac dl$ and $\frac{\rho d}{l}$.

Notice the sum in \eqref{eq:GOevennevenjoddsumswpowers2} is over the set $P$ of pairs $(d,l)$ with $2^{\nu_2(n)}d$ dividing $n$ and $l$ square free dividing $2^{\nu_2(n)}d$. Since the summand only depends on $\frac{d}{l}$, we introduce an equivalence relation on $P$ by declaring $(d,l) \sim (d',l')$ when $\frac{d}{l} = \frac{d'}{l'}$. Now, we need to identify the parity of cardinality of the $\sim$ classes, since those with even cardinality will not contribute to \eqref{eq:GOevennevenjoddsumswpowers2}. Suppose $(d,l)$ in $P$ with $l$ odd. Then we can divide both $l$ and $d$ by $l$, obtaining a new pair $(d',l')$ with the same ratio $\frac{d}{l} = \frac{d'}{l'}$ and $l'=1$. The number of pairs with this given ratio is then equal to $2^m$ where $m$ is the number of distinct odd prime factors of $n/(d')$. It follows that all these equivalence classes have even cardinality, except for the class of $(n,1)$. Now suppose $(d,l)$ in $P$ with $l $ even. Then we can divide both $l$ and $d$ by $l/2$, obtaining an equivalence pair $(d',2)$. The number of pairs equivalent to $(d',2)$ is then $2^m$ where $m$ is the number of odd prime factors of $n/d'$. Thus, all these equivalence classes have even cardinality, except for the class of $(n,2)$. However, for $(d,l)=(n,2)$, the binomial coefficient $\binom{\frac{2^{\nu_2(n)}d}{l}-1}{\frac{\rho 2^{\nu_2(n)}d}{l}-1} = \binom{\frac{n}{2}-1}{\frac{j}{2}-1}$, which is $0$ because $\frac{j}{2}$ is not an integer and by convention (see Theorem \ref{thm:lucastheorem}), binomial coefficients with fractional lower terms are $0$. Thus,

\[
|GO_{even}(L/k,j)| \equiv \binom{n-1}{j-1}\mod2.
\]
By Lucas's theorem (Theorem~\ref{thm:lucastheorem}), we have $|GO_{even}(L/k,j)| $ is odd if and only if $j-1 \prec n-1$. Finally, because $n$ is even and $j$ is odd, it follows that $j-1 \prec n-1$ if and only if $j-1 \prec n-2$. Then since both $j-1$ and $n-2$ are even, we have $j-1 \prec n-2$ if and only if $\frac{j-1}{2} \prec \frac{n-2}{2}$, which gives the desired formula.

\end{proof}

By Lemma \ref{lem:deltaGoeven}, Proposition \ref{prop:case2} implies Theorem \ref{thm:closedformula-1} for the case $\nu_2(n)\geq1$ and $\nu_2(j)=0$.

\subsection{Combinatorics of symmetric orbits} \label{subsection:comb1}

We will mainly work on the necklace interpretation for the rest of the paper. Note that $C_2$ generated by the flip $f$ acts on the set $\Orb(C_n, \mathrm{Neck}(n,j)$. For a set $S\subset \SN$, let $S^f$ denote the subset of $S$ that is fixed by the flipping action. If $l$ is a necklace with orbit $[l] \in \Orb(C_n, \mathrm{Neck}(n,j))^f$, then $fl= r^m l$ for some $m\in\bZ$. This implies that $l$ has an axis of symmetry: the axis rotated counterclockwise by $\frac{m\pi}{n}$ from the vertical axis. As we will see in a moment, a symmetry axis will allow us to decompose a necklace. Therefore, it will be important to consider orbits under rotation of pairs consisting of a necklace and a symmetry axis.  

Consider the 2-dimensional faithful representation $\Phi:D_n\hookrightarrow O(2,\bR)$,
\[
\Phi(r) = \left(\begin{array}{cc}
    \cos(\frac{2\pi }{n}) & -\sin(\frac{2\pi }{n}) \\
     \sin(\frac{2\pi }{n})& \cos(\frac{2\pi }{n})
\end{array}\right), \quad \Phi(f) = \left(\begin{array}{cc}
    -1 & 0 \\
     0 & 1
\end{array}\right).
\] Fix $ x\in \bR^2,x = (0,1)$, then $l\in\Neck(n,j)$ can be represented by a subset of size $j$ in $\Phi(D_n)\cdot x$. If $[l]\in \SN^f$, then for every representative $l\in \Neck(n,j)$ there exists $\sigma\in\bP^1(\bR)$ such that $l$ is invariant under the reflection $f_\sigma\in O(2,\bR)$ with respect to $\sigma$. Any such linear space $\sigma$ is called a symmetry axis of $l$.
A symmetry axis $[(l,\sigma)]$ for $[l]\in \SN^f$ is the $C_n$-orbit of the pair $(l,\sigma)$. 

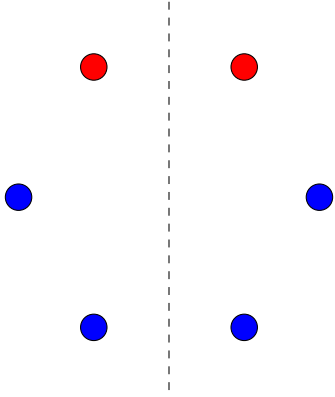
\begin{figure}
\centering
\begin{tikzpicture}

    \def\radius{2cm}

    \def\numBeads{6}

    \foreach \i in {1,...,\numBeads} {
       
        \pgfmathsetmacro{\angle}{360/\numBeads * (\i+1)}

        \ifnum\i=2
            \node[draw, fill=blue, circle, minimum size=10pt] at (\angle:\radius) {};
        \else
            \ifnum\i=3
                \node[draw, fill=blue, circle, minimum size=10pt] at (\angle:\radius) {};
            \else
                \ifnum\i=4
                    \node[draw, fill=blue, circle, minimum size=10pt] at (\angle:\radius) {};
                \else
                    \ifnum\i=5
                        \node[draw, fill=blue, circle, minimum size=10pt] at (\angle:\radius) {};
                    \else
                        \node[draw, fill=red, circle, minimum size=10pt] at (\angle:\radius) {};
                    \fi
                \fi
            \fi
        \fi
    }

    \draw[dashed] (0,1.3*\radius) -- (0,-1.3*\radius);

\end{tikzpicture}
\caption{An element in $\Orb(C_6,\Neck(6,4))^f$ with a unique symmetry axis.}
\end{figure}

\begin{definition}
    For any $[l]\in \SN$, define the period of $[l]$ as $\pi([l]):= |C_n\cdot l| $, which is well defined as $|C_n\cdot l|$ is independent from the choice of representative of $[l]$.
\end{definition}

\begin{definition}
For two symmetry axes  $[(l_1,\sigma_1)]$ and $[(l_2,\sigma_2)]$ of $[l]\in\SN^f$, define their distance as
    \[
    d([(l_1,\sigma_1)],[(l_2,\sigma_2)]) = \min\{|m|:r^m\cdot \sigma_1 =\sigma_2,(l_i,\sigma_i)\in [(l_i,\sigma_i)],i=1,2 \},
    \]
    where $m$ is allowed to be half-integers in the sense that $\Phi(r)^{1/2}$ is the counterclockwise rotation by $\frac{2 \pi i}{2n}$.    
\end{definition}
    
\begin{lemma}\label{lem:2axes}
    Let $[l]$  be in $\SN^f$. If $\frac{n}{\pi([l])}$ is odd, then there is a unique symmetry axis. Otherwise, there are two distinct symmetry axes $[(l_1,\sigma_1)]$ and $[(l_2,\sigma_2)]$ with $d([(l_1,\sigma_1)],[(l_2,\sigma_2)]) = \frac{\pi([l])}{2}$.
\end{lemma}

\begin{proof}
 Let $[(l_1,\sigma_1)]$ and $[(l_2,\sigma_2)]$ be two symmetry axes of $[l]$.
    As the composition $f_{\sigma_2}\circ f_{\sigma_1}\in \Phi(C_n)$ is $r$ raised to the power of $2d([(l_1,\sigma_1)],[(l_2,\sigma_2)])$, we have \[\pi([l])~ |~2d([(l_1,\sigma_1)],[(l_2,\sigma_2)]).\] On the other hand, by definition $d([(l_1,\sigma_1)],[(l_2,\sigma_2)])<\pi([l])$. Thus, $d([(l_1,\sigma_1)],[(l_2,\sigma_2)]) = \frac{\pi([l])}{2}$ or $0$. Moreover, in the first case, we must have that $\frac{n}{\pi([l])}$ is even.
\end{proof}

For example, Figure \ref{fig:twosymmetryaxes} shows an element $[l]\in\Orb(C_6,\Neck(6,4))^f$ with two different symmetry axes. In this example, $\pi([l]) =3$, from which $\frac{6}{\pi([l])} = 2$, and the distance between the two axes of symmetry is $\frac{3}{2}$.

\begin{definition}
A symmetry axis $[(l,\sigma)]$ for $[l]\in \SN^f$ is called {\em type 1} if $\sigma$ does not intersect $\Phi(D_n)\cdot x$. It is called {\em type 2} if $\sigma$ intersects at least one element of $\Phi(D_n)\cdot x$.    
\end{definition}
For $[l]\in\SN^f$, if $n$ is odd, then the only symmetry axis of $[l]$  is of type 2. If $n$ is even, then $[l]$ has two symmetry axes, one for each type. As noted above, an example for $n$ is even shown in Figure \ref{fig:twosymmetryaxes}.

\begin{figure}
\centering
\begin{tikzpicture}
    \def\radius{2cm}

    \def\numBeads{6}

    \foreach \i in {1,...,\numBeads} {
        \pgfmathsetmacro{\angle}{360/\numBeads * (\i + 1)}
        
        \ifnum\i=1
            \node[draw, fill=blue, circle, minimum size=10pt] at (\angle:\radius) {};
        \else
            \ifnum\i=3
                \node[draw, fill=blue, circle, minimum size=10pt] at (\angle:\radius) {};
            \else
                \ifnum\i=4
                    \node[draw, fill=blue, circle, minimum size=10pt] at (\angle:\radius) {};
                \else
                    \ifnum\i=6
                        \node[draw, fill=blue, circle, minimum size=10pt] at (\angle:\radius) {};
                    \else
                        \node[draw, fill=red, circle, minimum size=10pt] at (\angle:\radius) {};
                    \fi
                \fi
            \fi
        \fi
    }

    \draw[dashed] (0,1.3*\radius) -- (0,-1.3*\radius);

    \draw (180:1.3*\radius) -- (0:1.3*\radius);
\end{tikzpicture}
\caption{An element in $\Orb(C_6,\Neck(6,4))^f$ with two symmetry axes. The dashed line is a symmetry axis of type 1 and the solid line is of type 2.}\label{fig:twosymmetryaxes}
\end{figure}
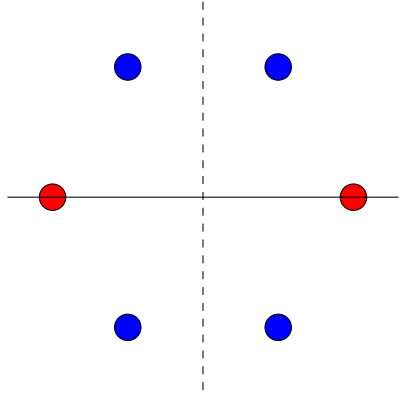

 For  $S\subset \SN^f$, denote $S_i$ as the subset of $S$ that has a symmetry axis of type $i$, for $i=1,2$. We have
   $\SN^{f} = \SN_1^{f}\bigcup \SN_2^{f}$. 

\begin{lemma}\label{lem:intersection}
    For $\nu_2(n)\geq 1$,
    \[\SN_1^{f}\cap \SN_2^{f} =  \SoN^{f}.\] In particular,
    \[
    \SeN_1^{f}\cap \SeN_2^{f}=\emptyset.
    \]
\end{lemma}
\begin{proof}
    If $[l]\in \SN^f$ has two axes  $[(l_1,\sigma_1)],[(l_2,\sigma_2)]$ of different type, then $d([(l_1,\sigma_1)],[(l_2,\sigma_2)])$ is an half integer. By Lemma \ref{lem:2axes}, we have $\pi([l]) = 2d([(l_1,\sigma_1)],[(l_2,\sigma_2)])$ is odd. 
\end{proof}
Consider again the necklace in Figure \ref{fig:twosymmetryaxes}. This necklace has period $3$, and is contained in both sides of the first equation in Lemma \ref{lem:intersection}. An immediate corollary of Lemma~\ref{lem:intersection} is the following.

\begin{corollary} For $\nu_2(n)\geq 1$, we have
    \begin{align*}
        |\SeN| \equiv & |\SeN^f| \mod 2\\
         = & |\SeN^f_1| +|\SeN^f_2|. \\
    \end{align*}
\end{corollary}
\subsection{Combinatorics of type 1 symmetric orbits} \label{subsection:comb2}

Now assume $n$ is even. By choosing every other bead of a representative necklace $l$, any $[l]\in \SN$ can be uniquely decomposed into an unordered pair of necklace orbits (under the rotation action) with $\frac n2$ beads. This constructs a well-defined map
    \begin{equation}\label{defn:phi}
    \begin{split}
        \phi:\SN&\to   \bigcup_{j_1+j_2 = j}\mathrm{Sym}(\Orb(C_{\frac{n}{2}},\Neck(\frac n 2,j_1)), \Orb(C_{\frac{n}{2}},\Neck(\frac n 2,j_2)))\\
        [l]& \mapsto ([l_1],[l_2]),
        \end{split}
    \end{equation}
where $\mathrm{Sym}$ denotes the symmetric product.
    
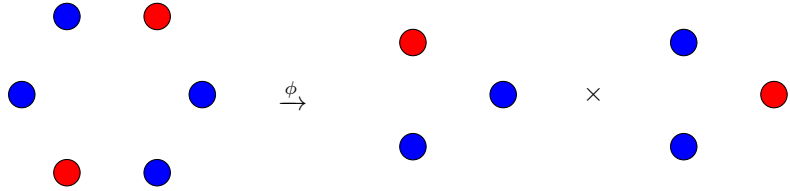
\begin{figure}
\centering
\begin{tikzpicture}[scale=0.8]

    \begin{scope}[shift={(0, 0)}]

        \def\radius{1.5cm}

        \def\numBeads{6}

        \foreach \i in {1,...,\numBeads} {
            \pgfmathsetmacro{\angle}{360/\numBeads * (\i - 1)}
            
            \ifnum\i=1
                \node[draw, fill=blue, circle, minimum size=10pt] at (\angle:\radius) {};
            \else
                \ifnum\i=3
                    \node[draw, fill=blue, circle, minimum size=10pt] at (\angle:\radius) {};
                \else
                    \ifnum\i=4
                        \node[draw, fill=blue, circle, minimum size=10pt] at (\angle:\radius) {};
                    \else
                        \ifnum\i=6
                            \node[draw, fill=blue, circle, minimum size=10pt] at (\angle:\radius) {};
                        \else
                            \node[draw, fill=red, circle, minimum size=10pt] at (\angle:\radius) {};
                        \fi
                    \fi
                \fi
            \fi
        }
    \end{scope}

    \node at (3, 0) { $\xrightarrow{\phi}$};

    \begin{scope}[shift={(5.5, 0)}]

        \def\radius{1cm}

        \def\numBeads{3}

        \foreach \i in {1,...,\numBeads} {
            \pgfmathsetmacro{\angle}{360/\numBeads * (\i - 1)}
            
            \ifnum\i=1
                \node[draw, fill=blue, circle, minimum size=10pt] at (\angle:\radius) {};
            \else
                \ifnum\i=3
                    \node[draw, fill=blue, circle, minimum size=10pt] at (\angle:\radius) {};
                \else
                    \node[draw, fill=red, circle, minimum size=10pt] at (\angle:\radius) {};
                \fi
            \fi
        }
    \end{scope}

    \node at (8, 0) { $\times$};

    \begin{scope}[shift={(10, 0)}]

        \def\radius{1cm}

        \def\numBeads{3}

        \foreach \i in {1,...,\numBeads} {
            \pgfmathsetmacro{\angle}{360/\numBeads * (\i - 1)}
            
            \ifnum\i=2
                \node[draw, fill=blue, circle, minimum size=10pt] at (\angle:\radius) {};
            \else
                          \ifnum\i=3
                \node[draw, fill=blue, circle, minimum size=10pt] at (\angle:\radius) {};
                \else
                \node[draw, fill=red, circle, minimum size=10pt] at (\angle:\radius) {};
            \fi
            \fi
        }
    \end{scope}

\end{tikzpicture}
\caption{An element in $\Orb(C_6,\Neck(6,4))^f$ decomposes to a symmetric product in $\Orb(C_3,\Neck(3,2)^f$ under $\phi$.}
\end{figure}
    \begin{lemma}
        For $[l]\in \SN_1^f$, we have $\phi([l])=([l_1], f\cdot [l_1])$, where $l_1\in \Neck(\frac{n}{2},\frac j2)$.
    \end{lemma}
    \begin{proof}
    For $[l]\in \SN^f_1$, let $[(l,\sigma)]$ denote a symmetry axis of type 1. The action of $f$ on $\SN^f_1$ can be induced by the reflection $f_\sigma$.  Since $f_\sigma$ exchanges the two smaller necklaces, we must have $[l_2]=f\cdot[l_1]$.
    \end{proof}
In particular, when restricted to $\SN_1^f$, $\phi$ surjects to $ \mathrm{Sym}(\SNhalf,f\cdot \SNhalf$. Therefore, we have
    \[
    \SN^f_1 = \bigcup_{([l],f\cdot[l]))\in\mathrm{Sym}(\SNhalf,f\cdot \SNhalf}\phi^{-1}([l],f\cdot[l]).
    \]
\begin{lemma}\label{lem:fibperiod1}
Consider the restriction of $\phi$ to $\SN_1^f$. Then
    \begin{enumerate}
    \item $|\phi^{-1}([l],f\cdot [l])| = \pi([l])$, if $[l]\notin \SNhalf^f$, 
    \item  $|\phi^{-1}([l],[l])| = \frac{\pi([l])+1}{2}$ if $[l]\in \SNhalf^f$ and $\pi([l])\equiv 1\mod 2$, 
    \item $|\phi^{-1}([l], [l])| = \frac{\pi([l])}{2}$ if $[l]\in \SNhalf^f$ and $\pi([l])\equiv 0\mod 2$.
    \end{enumerate}
\end{lemma}
\begin{proof}
    $|\phi^{-1}([l],f\cdot [l])|$ is the number of elements in $\SN_1^f$ that can be formed by interweaving $[l]$ and $f\cdot [l]$. The above statement can be seen by considering a fixed representative of $[l]$, and counting how many different ways one can insert $f\cdot l$.
\end{proof}

For example, the element shown in Figure \ref{fig:orb52} has period 5, which can generate 3 distinct elements by  (2)  of Lemma \ref{lem:fibperiod1}. These three elements are shown in Figure \ref{fig:orb104}.

\begin{figure}
\centering
\begin{tikzpicture}
    \def\radius{1.5cm}

    \def\numBeads{5}

    \foreach \i in {1,...,\numBeads} {
        \pgfmathsetmacro{\angle}{360/\numBeads * (\i+1)}
        
        \ifnum\i=2
            \node[draw, fill=blue, circle, minimum size=10pt] at (\angle:\radius) {};
        \else
            \ifnum\i=3
                \node[draw, fill=blue, circle, minimum size=10pt] at (\angle:\radius) {};
                    \else
                        \node[draw, fill=red, circle, minimum size=10pt] at (\angle:\radius) {};
            \fi
        \fi
    }

\end{tikzpicture}
\caption{An element in $\Orb(C_5,\Neck(5,2))^f$.}\label{fig:orb52}
\end{figure}
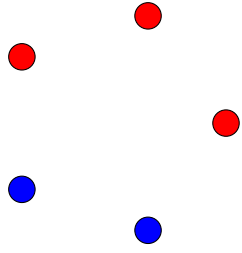

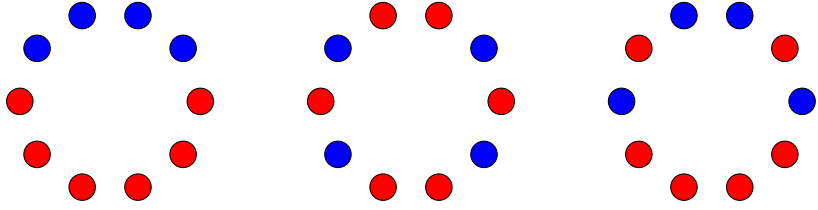
\begin{figure}
\centering
\begin{tikzpicture}[scale=0.8]

    \begin{scope}[shift={(0, 0)}]
        \def\radius{1.5cm}

        \def\numBeads{10}

        \foreach \i in {1,...,\numBeads} {
            \pgfmathsetmacro{\angle}{360/\numBeads * (\i - 1)}
            
            \ifnum\i=2
                \node[draw, fill=blue, circle, minimum size=10pt] at (\angle:\radius) {};
            \else
                \ifnum\i=3
                    \node[draw, fill=blue, circle, minimum size=10pt] at (\angle:\radius) {};
                \else
                    \ifnum\i=4
                        \node[draw, fill=blue, circle, minimum size=10pt] at (\angle:\radius) {};
                    \else
                        \ifnum\i=5
                            \node[draw, fill=blue, circle, minimum size=10pt] at (\angle:\radius) {};
                        \else
                            \node[draw, fill=red, circle, minimum size=10pt] at (\angle:\radius) {};
                        \fi
                    \fi
                \fi
            \fi
        }
    \end{scope}

    \begin{scope}[shift={(5, 0)}]
        \def\radius{1.5cm}

        \def\numBeads{10}

        \foreach \i in {1,...,\numBeads} {
            \pgfmathsetmacro{\angle}{360/\numBeads * (\i+2)}
            
            \ifnum\i=2
                \node[draw, fill=blue, circle, minimum size=10pt] at (\angle:\radius) {};
            \else
                \ifnum\i=4
                    \node[draw, fill=blue, circle, minimum size=10pt] at (\angle:\radius) {};
                \else
                    \ifnum\i=7
                        \node[draw, fill=blue, circle, minimum size=10pt] at (\angle:\radius) {};
                    \else
                        \ifnum\i=9
                            \node[draw, fill=blue, circle, minimum size=10pt] at (\angle:\radius) {};
                        \else
                            \node[draw, fill=red, circle, minimum size=10pt] at (\angle:\radius) {};
                        \fi
                    \fi
                \fi
            \fi
        }
    \end{scope}

    \begin{scope}[shift={(10, 0)}]
        \def\radius{1.5cm}

        \def\numBeads{10}

        \foreach \i in {1,...,\numBeads} {

            \pgfmathsetmacro{\angle}{360/\numBeads * (\i - 1)}

            \ifnum\i=1
                \node[draw, fill=blue, circle, minimum size=10pt] at (\angle:\radius) {};
            \else
                \ifnum\i=3
                    \node[draw, fill=blue, circle, minimum size=10pt] at (\angle:\radius) {};
                \else
                    \ifnum\i=4
                        \node[draw, fill=blue, circle, minimum size=10pt] at (\angle:\radius) {};
                    \else
                        \ifnum\i=6
                            \node[draw, fill=blue, circle, minimum size=10pt] at (\angle:\radius) {};
                        \else
                            \node[draw, fill=red, circle, minimum size=10pt] at (\angle:\radius) {};
                        \fi
                    \fi
                \fi
            \fi
        }
    \end{scope}

\end{tikzpicture}
\caption{The three distinct elements in $\Orb(C_{10},\Neck(10,4))^f$ that can be generated by the same element in $\Orb(C_5,\Neck(5,2))^f$.}\label{fig:orb104}
\end{figure}

\begin{lemma}\label{lem:fibperiod2}
Take $[l]\in \SN_1^f$, then we have $\phi([l])=([l_1],[l_1])$ for some $[l_1]\in\SNhalf^f$. If $[l_1]$ satisfies $\pi([l_1])\equiv 1 \mod 2$, then $\pi([l]) = \pi([l_1])$. Otherwise, $\pi([l]) = 2\pi([l_1])$.

\end{lemma}
\begin{proof}
    The exceptional case is the following. View $l_1$ as a subset of $\Phi(D_n)\cdot x$. Then flip $l_1$ with respect to $\sigma_0\in \bP^1(\bR)$, where $\sigma_0$ is perpendicular to a symmetry axis of $l_1$. Then $[l]$ is the $C_n$-orbit of $l_1\cup f_{\sigma_0}\cdot l_1$.
\end{proof}

For example, the three elements shown in Figure \ref{fig:orb104} have period $10,5,10$ respectively, where the middle one is the exceptional case in Lemma \ref{lem:fibperiod2}.

Therefore, for $[l]\in \Orb_{odd}(C_{\frac n2},\Neck(\frac n2,\frac j2))^f$, we have 
\begin{equation}\label{eq:evenfiberperiod}
\left|\phi^{-1}([l],[l])\cap  \SeN\right|= \frac{\pi([l])-1}{2}.    
\end{equation}

\begin{proposition}\label{prop:S1}
For $\nu_2(n)\geq 1$,    \[|\SeN_1^f| = \frac12 \left(\binom{\frac n2}{\frac j2}-|\SoNhalf^f|\right).\]
\end{proposition}
\begin{proof} By Lemma \ref{lem:fibperiod1} and \eqref{eq:evenfiberperiod}, we have
\begin{align*}
    |\SeN_1^f| = & \frac 12\sum_{[l]\notin \SNhalf^f}\pi([l]) + \sum_{[l]\in \SeNhalf^f}\frac{\pi([l])}{2}\\
    & + \sum_{[l]\in\SoNhalf^f}\frac{\pi([l])-1}{2}\\
     = & \frac12 \left(\sum_{[l]\in \SNhalf}\pi([l]) - |\SoNhalf^f|\right)
\end{align*}
By definition $\sum_{[l]\in \SNhalf}\pi([l]) = |\mathrm{Neck}(\frac n2,\frac j2)| = \binom{\frac n2}{\frac j2}$. So 
\begin{align*}
  |\SeN_1^f|   =  \frac12\left(\sum_{d|\frac n2}dN(d,\rho d)-|\SoNhalf^f|\right)
\end{align*}
\end{proof}

\begin{lemma}\label{lem:ooddvalue}
For any $n,j\in \bZ_{\geq0}$, we have
    $$|\SoN^f| = \left\{ \begin{array}{cc}
        \binom{\frac{n}{2^{\nu_2(j)+1}}-\frac12}{\frac{j}{2^{\nu_2(j)+1}}}, & \textrm{ when } \nu_2(j)>\nu_2(n),\\ \\
         \binom{\frac{n}{2^{\nu_2(j)+1}}-\frac12}{\frac{j}{2^{\nu_2(j)+1}}-\frac12}, &\textrm{ when }  \nu_2(j) = \nu_2(n),\\ \\
         0, & \textrm{ when } \nu_2(j)<\nu_2(n).
    \end{array}\right.$$
\end{lemma}
\begin{proof}
    Notice that for any $n,j\in 2\bZ$, we have $$\SoN^f = \SoNhalf^f.$$ Then, the problem reduces to the case when $\nu_2(n)\cdot\nu_2(j) =0$. In this case, the counting corresponds to the number of ways to assign exactly half of the blue beads to one half of the necklace when the total number of blue beads is even, or half minus one when the number is odd.
\end{proof}

\subsection{Proof of Theorem \ref{thm:closedformula-1} when $\nu_2(n)=1$ and $\nu_2(j)\geq1$}\label{subsection:case2}

Consider the restriction of the map $\phi$ in \eqref{defn:phi} to $\SN_2^f$. As $n\equiv 2\mod4$, the image is contained in \[\bigcup_{j_1+j_2=j}\mathrm{Sym}(\Orb(C_{\frac n2},\Neck(\frac n2,j_1))^f,\Orb(C_{\frac n2},\Neck(\frac n2,j_2))^f).\]
An example of the morphism $\phi$ in this case is shown in Figure \ref{fig:examplephi2}.

\begin{figure}
\centering
\begin{tikzpicture}[scale=0.8]

    \begin{scope}[shift={(0, 0)}]
        \def\radius{1.5cm}

        \def\numBeads{10}

        \foreach \i in {1,...,\numBeads} {
            \pgfmathsetmacro{\angle}{360/\numBeads * (\i - 1)+90}
            
            \ifnum\i=1
                \node[draw, fill=blue, circle, minimum size=10pt] at (\angle:\radius) {};
            \else
                \ifnum\i=2
                    \node[draw, fill=blue, circle, minimum size=10pt] at (\angle:\radius) {};
                \else
                    \ifnum\i=10
                        \node[draw, fill=blue, circle, minimum size=10pt] at (\angle:\radius) {};
                    \else
                        \ifnum\i=6
                            \node[draw, fill=blue, circle, minimum size=10pt] at (\angle:\radius) {};
                        \else
                            \node[draw, fill=red, circle, minimum size=10pt] at (\angle:\radius) {};
                        \fi
                    \fi
                \fi
            \fi
        }
    \end{scope}
    \node at (3, 0) { $\xrightarrow{\phi}$};

    \begin{scope}[shift={(5.5, 0)}]

        \def\radius{1cm}

        \def\numBeads{5}

        \foreach \i in {1,...,\numBeads} {
            \pgfmathsetmacro{\angle}{360/\numBeads * (\i - 1)+90}
            
            \ifnum\i=1
                \node[draw, fill=blue, circle, minimum size=10pt] at (\angle:\radius) {};
            \else
                    \node[draw, fill=red, circle, minimum size=10pt] at (\angle:\radius) {};
                \fi
        }
    \end{scope}

    \node at (8, 0) { $\times$};

    \begin{scope}[shift={(10, 0)}]
        \def\radius{1cm}

        \def\numBeads{5}

        \foreach \i in {1,...,\numBeads} {
            \pgfmathsetmacro{\angle}{360/\numBeads * (\i - 1)-90}
            
            \ifnum\i=4
                \node[draw, fill=blue, circle, minimum size=10pt] at (\angle:\radius) {};
            \else
                \ifnum\i=3
                \node[draw, fill=blue, circle, minimum size=10pt] at (\angle:\radius) {};
                \else
                \ifnum\i=1
                \node[draw, fill=blue, circle, minimum size=10pt] at (\angle:\radius) {};
                \else
    
                \node[draw, fill=red, circle, minimum size=10pt] at (\angle:\radius) {};
            \fi
            \fi
            \fi
        }
    \end{scope}

\end{tikzpicture}
\caption{An element in $\Orb(C_{10},\Neck(10,4))^f_2$ decomposes under $\phi$ to a symmetric product of two elements in $\Orb(C_{5},\Neck(5,1))^f$ and $\Orb(C_{5},\Neck(5,3))^f$, repsectively.}\label{fig:examplephi2}
\end{figure}
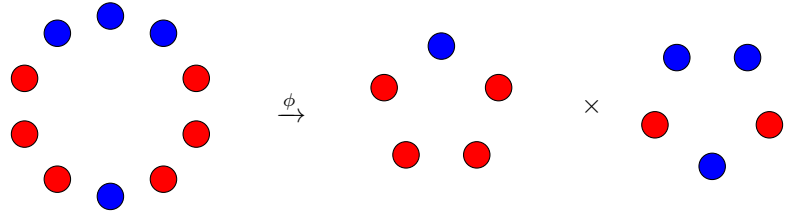

Moreover, since $\frac{n}{2}$ is odd, and thus any symmetry axis for $[l]$ only passes through one bead, it is of type two. It follows that the restriction of $\phi$ to $\SN_2^f$ is bijective, and we have
\begin{equation}\label{eqn:ob2}
\begin{split}
    |\SN_2^f|  = & \sum_{j_1+j_2=j,j_1< j_2}|\Orb(C_{\frac n2},\Neck(\frac n2,j_1))^f|\cdot |\Orb(C_{\frac n2},\Neck(\frac n2,j_2))^f| \\
    & + \frac12|\SNhalf^f|\cdot(|\SNhalf^f|+1).\\
    \end{split}
    \end{equation}
    Next, we have 
    \begin{lemma} Let $[l]\in \SN_2^f$. Then
    \begin{enumerate}
        \item $\pi([l]) = \pi([l_1])$,  \text{ if } $\phi([l]) = ([l_1], [l_1])$,
        \item $\pi([l]) = 2\mathrm{lcm}(\pi([l_1]),\pi([l_2]))$, \text{ if } $\phi([l]) = ([l_1],[l_2])$, $[l_1]\neq  [l_2]$.
    \end{enumerate}    
    \end{lemma}
   
    \begin{proof}
        $l$ has a stabilizer $r^k$ with $k\equiv 1\mod2$ under the $C_n$ action if and only if  $\phi([l]) = ([l_1],[l_1])$. Therefore, when $\phi([l]) = ([l_1],[l_1])$, the period of $[l]$ is exactly that of $[l_1]$. For the other case, the reason why there is a multiplier $2$ in front of $\mathrm{lcm}(\pi([l_1]),\pi([l_2]))$ is because any consecutive $\pi([l])$ number of beads consists of $\frac{\mathrm{lcm}(\pi([l_1]),\pi([l_2]))}{\pi([l_2])}$ copies of $[l_1]$ and $\frac{\mathrm{lcm}(\pi([l_1]),\pi([l_2]))}{\pi([l_1])}$ copies of $[l_2]$.
        \end{proof}
     Figure \ref{fig:examplephi2} provides an example of the second case. Since $\frac n2$ is odd, we have
    \[
    |\SoN^f| = |\SoNhalf^f|.
    \]
    Then, as a simple consequence,
    \begin{equation}\label{eqn:ob3}
    \begin{split}
        |\SeN^f_2| & = |\SN_2^f| - |\SoN^f| \\
        & = |\SN_2^f| - |\SoNhalf^f|.
    \end{split}
    \end{equation}
    \begin{proposition}\label{prop:oevenfixvalue}
      \begin{align*}
        |\SeN^f|  
        =  \left\{
        \begin{array}{cc}
            \binom{\frac n2}{\frac j2} - \binom{\frac{n-2}{4}}{\frac{j-2}{4}}, & j\equiv 2\mod4, \\
           \binom{\frac n2}{\frac j2}-\binom{\frac{n-2}{4}}{\frac{j}{4}},  & j\equiv 0\mod4.
        \end{array}\right.
    \end{align*}   
    \end{proposition}
    \begin{proof}

    Since $\frac n2$ is odd,  we have $\SoNhalf^f = \SNhalf^f$. Therefore, combining \eqref{eqn:ob2}, \eqref{eqn:ob3}, and Proposition \ref{prop:S1}, and noting Lemma \ref{lem:intersection}, we have
    \begin{align*}
        |\SeN^f|  = & |\SeN_1^f| + |\SeN_2^f| \\
        = &\frac{1}{2} \sum_{j_1+j_2=j,j_1< j_2}|\Orb(C_{\frac n2},\Neck(\frac n2,j_1))^f|\cdot |\Orb(C_{\frac n2},\Neck(\frac n2,j_2))^f|\\
        &+\frac 12\left(\binom{\frac n2}{\frac j2} -|\SNhalf^f|\right).
    \end{align*}

    Since $j$ is even, we have $j_1\equiv j_2\mod 2$. As $\frac{n}{2}$ is odd, we further have $\nu_2(j_i)>\nu_2(n)$, $i=1,2$ or $\nu_2(j_i)=\nu_2(n)$, $i=1,2$. Thus by Lemma \ref{lem:ooddvalue}, we have
    \begin{align*}
        & \frac{1}{2} \sum_{j_1+j_2=j,j_1< j_2}|\Orb(C_{\frac n2},\Neck(\frac n2,j_1))^f|\cdot |\Orb(C_{\frac n2},\Neck(\frac n2,j_2))^f|\\
        = &\frac12 \sum_{i=0}^{\frac j2-1}\binom{\frac{n-2}{4}}{i}\binom{\frac{n-2}{4}}{\frac j2-1-i}+\frac12\sum_{i=0}^{\frac j2}\binom{\frac{n-2}{4}}{i}\binom{\frac{n-2}{4}}{\frac j2-i}\\
        = & \frac12\left(\binom{\frac{n}{2}-1}{\frac j2-1}+\binom{\frac n2-1}{\frac j2}\right) = \frac12\binom{\frac n2}{\frac j2},
    \end{align*}
    where in the last line we have used the Pascal identity.  Thus
    \begin{align*}
        |\SeN^f|  = & \frac12\left(\binom{\frac n2}{\frac j2}+\binom{\frac n2}{\frac j2}\right)-|\SNhalf^f|\\
        = & \left\{
        \begin{array}{cc}
            \binom{\frac n2}{\frac j2} - \binom{\frac{n-2}{4}}{\frac{j-2}{4}}, & j\equiv 2\mod4, \\
           \binom{\frac n2}{\frac j2}-\binom{\frac{n-2}{4}}{\frac{j}{4}},  & j\equiv 0\mod4.
        \end{array}\right.
    \end{align*}
        \end{proof}  
    \begin{proposition}\label{prop:casenu2n1nu2jgeq1}
     For $\nu_2(n)=1$ and $\nu_2(j)\geq1$, we have  $\Delta(n,j) =0$.
\end{proposition}

\begin{proof}

    By the Corollary \ref{cor:corKummer1} of Kummer's theorem, for $j\equiv 2\mod4$, we have
    \[
    \nu_2\binom{\frac{n-2}{4}}{\frac{j-2}{4}} = \nu_2\binom{\frac{n}{2}-1}{\frac j2-1}=\nu_2\binom{\frac{n}{2}}{\frac j2},
    \]
    and for $j\equiv 0\mod4$,
    \[
    \nu_2\binom{\frac{n-2}{4}}{\frac{j}{4}} = \nu_2\binom{\frac{n}{2}-1}{\frac j2}=\nu_2\binom{\frac{n}{2}}{\frac j2}.
    \]
    By Proposition \ref{prop:oevenfixvalue},
    \[
    |\SeN^f|\equiv 0\mod2.
    \]
    Finally, by Lemma \ref{lem:deltaGoeven},
\[
\Delta(n,j) = (u-1)\cdot|\SeN^f| =0.
\]
\end{proof}

\subsection{Proof of Theorem \ref{thm:closedformula-1} when $\nu_2(n)\geq 2$ and $\nu_2(j)\geq 1$}\label{subsection:case3}
Since both $n$ and $j$ are even, for any $[l]\in \SN_2^f$ there are two beads on any type 2 symmetry axis, and they have to be of the same color.   Denote $[l']$ as the subset formed by removing the two beads on a chosen type 2 symmetry axis for $l\in [l]$. If the two beads are blue, then
\[
[l']\in \Orb(C_{n-2},\Neck(n-2,j-2))_1^f.
\]
Otherwise, the two beads are red and
\[
[l']\in \Orb(C_{n-2},\Neck(n-2,j))_1^f.
\]

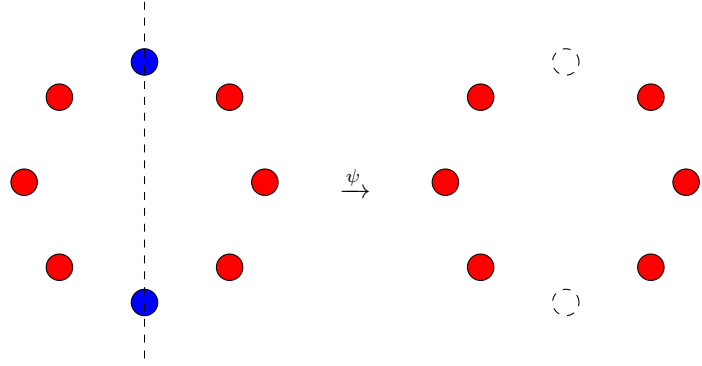
\begin{figure}
\centering

\begin{tikzpicture}[scale=0.8]
    \begin{scope}[shift={(0, 0)}]
        \def\radius{2cm}

        \def\numBeads{8}

        \foreach \i in {1,...,\numBeads} {
            \pgfmathsetmacro{\angle}{360/\numBeads * (\i +1)}
            
            \ifnum\i=1
                \node[draw, fill=blue, circle, minimum size=10pt] at (\angle:\radius) {};
            \else
                \ifnum\i=5
                    \node[draw, fill=blue, circle, minimum size=10pt] at (\angle:\radius) {};
                \else
                    \node[draw, fill=red, circle, minimum size=10pt] at (\angle:\radius) {};
                \fi
            \fi
        }

        \draw[dashed] (90:\radius*1.5) -- (270:\radius*1.5);
    \end{scope}
    \node at (3.5, 0) { $\xrightarrow{\psi}$};
    \begin{scope}[shift={(7, 0)}]
        \def\radius{2cm}

        \def\numBeads{8}

        \foreach \i in {1,...,\numBeads} {
            \pgfmathsetmacro{\angle}{360/\numBeads * (\i - 1)}
            
            \ifnum\i=3
                \node[draw, dashed, circle, minimum size=10pt] at (\angle:\radius) {};
            \else
                \ifnum\i=7
                    \node[draw, dashed, circle, minimum size=10pt] at (\angle:\radius) {};
                \else
                    \node[draw, fill=red, circle, minimum size=10pt] at (\angle:\radius) {};
                \fi
            \fi
        }
    \end{scope}

\end{tikzpicture}
\caption{By removing the two beads on the symmetry axis, $\psi$ maps an element in $\Orb(C_{8},\Neck(8,2))^f_2$ to $\Orb(C_{6},\Neck(6,0))^f_1$.}\label{fig:examplepsi}
\end{figure}

Now, if we start with $[l']\in \Orb(C_{n-2},\Neck(n-2,j-2))_1^f$, then we can recover an element in $\SN_2^f$ by picking a symmetry axis for $l'\in[l']$ and adding back two blue beads. We can also do a similar procedure for $[l']\in \Orb(C_{n-2},\Neck(n-2,j))_1^f$ by adding two red beads. However, the associated map
\[
\psi:\Orb(C_{n-2},\Neck(n-2,j-2))_1^f\amalg \Orb(C_{n-2},\Neck(n-2,j))_1^f \to \SN^f_2
\]
is surjective but not injective\footnote{For an example of how does the map $\psi$ works, see Figure \ref{fig:examplepsi}.
}.  In other words, if we count $\SN^f_2$ via counting the domain, then there will be over-counting. Figure \ref{fig:examplepsi} and Figure \ref{fig:examplepsi2} provide an example of the over-counting phenomenon. To understand the fibers of $\psi$, i.e., the over-counting, we first notice the following.
\begin{lemma}\label{lem:analysisofaxes}
    Every element $[l]$ in $\Orb(C_{n-2},\Neck(n-2,j-2))_1^f\amalg \Orb(C_{n-2},\Neck(n-2,j))_1^f$ has exactly one type 1 symmetry axis if $\pi([l])$ is even, and two symmetry axes of different type if $\pi([l])$ is odd.
\end{lemma}
\begin{proof}
Take $[l]\in \Orb(C_{n-2},\Neck(n-2,j-2))_1^f\amalg \Orb(C_{n-2},\Neck(n-2,j))_1^f]$. By Lemma \ref{lem:2axes}, if $\frac{n-2}{\pi([l])}$ is odd, then $[l]$ has exactly one symmetry axis, which has to pass between beads as $\pi([l])$ is even.
On the other hand, if $\frac{n-2}{\pi([l])}$ is even, since $n-2\equiv 2\mod 4$, $\pi([l])$ is odd, thus the two axes are of a different type, as their distance is a half-integer.
\end{proof}

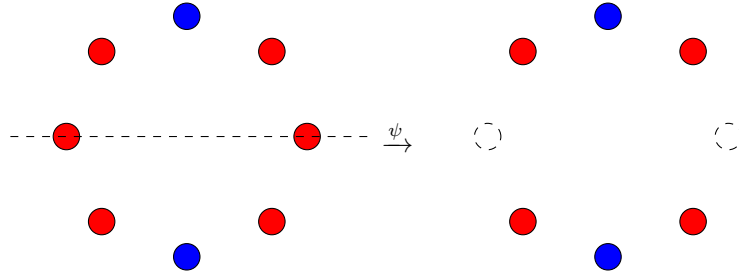
\begin{figure}
\centering

\begin{tikzpicture}[scale=0.8]
    \begin{scope}[shift={(0, 0)}]
        \def\radius{2cm}

        \def\numBeads{8}

        \foreach \i in {1,...,\numBeads} {
            \pgfmathsetmacro{\angle}{360/\numBeads * (\i +1)}
            
            \ifnum\i=1
                \node[draw, fill=blue, circle, minimum size=10pt] at (\angle:\radius) {};
            \else
                \ifnum\i=5
                    \node[draw, fill=blue, circle, minimum size=10pt] at (\angle:\radius) {};
                \else
                    \node[draw, fill=red, circle, minimum size=10pt] at (\angle:\radius) {};
                \fi
            \fi
        }

        \draw[dashed] (0:\radius*1.5) -- (180:\radius*1.5);
    \end{scope}
    \node at (3.5, 0) { $\xrightarrow{\psi}$};
    \begin{scope}[shift={(7, 0)}]
        \def\radius{2cm}

        \def\numBeads{8}

        \foreach \i in {1,...,\numBeads} {
            \pgfmathsetmacro{\angle}{360/\numBeads * (\i + 1)}
            
            \ifnum\i=3
                \node[draw, dashed, circle, minimum size=10pt] at (\angle:\radius) {};
            \else
                \ifnum\i=7
                    \node[draw, dashed, circle, minimum size=10pt] at (\angle:\radius) {};
             \else
                    \ifnum\i=5
                \node[draw, fill=blue, circle, minimum size=10pt] at (\angle:\radius) {};
            \else
                \ifnum\i=1
                    \node[draw, fill=blue, circle, minimum size=10pt] at (\angle:\radius) {};
                \else
                    \node[draw, fill=red, circle, minimum size=10pt] at (\angle:\radius) {};
                \fi
            \fi
            \fi
            \fi
        }
    \end{scope}

\end{tikzpicture}
\caption{The element in Figure \ref{fig:examplepsi} has another symmetry axis, which maps to  $\Orb(C_{6},\Neck(6,2))^f_1$ under $\psi$.}\label{fig:examplepsi2}
\end{figure}

The only elements $[l]\in \SN_2^f$ that have $|\psi^{-1}([l])|=2$ are the ones with two different axes both passing through two beads. Viewing on the domain, these are the elements in $\Orb(C_{n-2},\Neck(n-2,j-2))_1^f$ and  $\Orb(C_{n-2},\Neck(n-2,j))_1^f$ that have two axes of different types. By Lemma \ref{lem:analysisofaxes}, these are exactly the elements that have an odd period. 

The overcounting is not simply
\[
|\Orb_{odd}(C_{n-2},\Neck(n-2,j-2))_1^f|+| \Orb_{odd}(C_{n-2},\Neck(n-2,j))_1^f|
\]
as there would be an over-counting of the over-counting. This happens in the following situation. Denote the two symmetry axes for $[l]$ in the domain with $n-2$ beads as $[(l,\sigma_1)]$, $[(l,\sigma_2)]$, which is of type 1 and type 2, respectively. After adding the two beads along $[(l,\sigma_1)]$, the two symmetry axes $[(l,\sigma_1)]$, $[(l,\sigma_2)]$ will both become type 2 symmetry axis in $\psi([l])\in \SN^f_2$. Therefore, the over-counting of over-counting will happen if $[(l,\sigma_1)]$, $[(l,\sigma_2)]$ become the same axis in $\psi([l])$. A little thought combined with Lemma \ref{lem:2axes} shows this will happen if and only if $\psi([l])\in \SoN^f$. For example, if one wants to construct an element in $\Orb(C_8,\Neck(8,4))^f_2$ from the right-hand side of Figure \ref{fig:examplepsi2}, then one needs to add two blue beads to the vacant spots. This becomes a new symmetry axis. However, as shown in Figure \ref{fig:examplepsi3}, this axis is the same as the dashed one from the left-hand side.

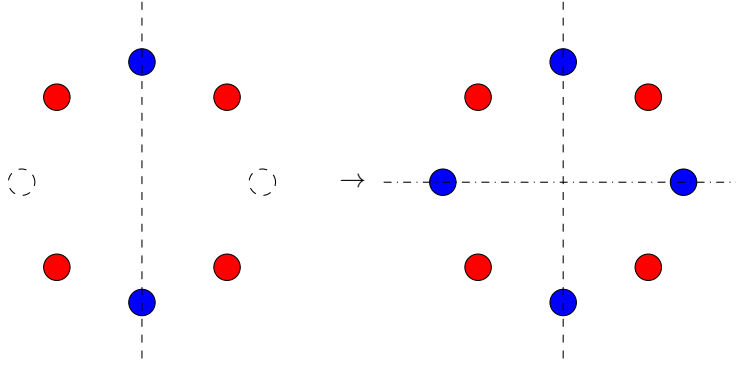
\begin{figure}
\centering

\begin{tikzpicture}[scale=0.8]
    \begin{scope}[shift={(0, 0)}]
        \def\radius{2cm}

        \def\numBeads{8}

        \foreach \i in {1,...,\numBeads} {

            \pgfmathsetmacro{\angle}{360/\numBeads * (\i + 1)}
            
            \ifnum\i=3
                \node[draw, dashed, circle, minimum size=10pt] at (\angle:\radius) {};
            \else
                \ifnum\i=7
                    \node[draw, dashed, circle, minimum size=10pt] at (\angle:\radius) {};
             \else
                    \ifnum\i=5
                \node[draw, fill=blue, circle, minimum size=10pt] at (\angle:\radius) {};
            \else
                \ifnum\i=1
                    \node[draw, fill=blue, circle, minimum size=10pt] at (\angle:\radius) {};
                \else
                    \node[draw, fill=red, circle, minimum size=10pt] at (\angle:\radius) {};
                \fi
            \fi
            \fi
            \fi
        }
\draw[dashed] (90:\radius*1.5) -- (270:\radius*1.5);

    \end{scope}
    \node at (3.5, 0) { $\to$};
    \begin{scope}[shift={(7, 0)}]

        \def\radius{2cm}

        \def\numBeads{8}

        \foreach \i in {1,...,\numBeads} {
            \pgfmathsetmacro{\angle}{360/\numBeads * (\i + 1)}
            
            \ifnum\i=3
                \node[draw, fill=blue, circle, minimum size=10pt] at (\angle:\radius) {};
            \else
                \ifnum\i=7
                    \node[draw, fill=blue, circle, minimum size=10pt] at (\angle:\radius) {};
             \else
                    \ifnum\i=5
                \node[draw, fill=blue, circle, minimum size=10pt] at (\angle:\radius) {};
            \else
                \ifnum\i=1
                    \node[draw, fill=blue, circle, minimum size=10pt] at (\angle:\radius) {};
                \else
                    \node[draw, fill=red, circle, minimum size=10pt] at (\angle:\radius) {};
                \fi
            \fi
            \fi
            \fi
        }
        \draw[dashed] (90:\radius*1.5) -- (270:\radius*1.5);
\draw[dash dot] (0:\radius*1.5) -- (180:\radius*1.5);

    \end{scope}

\end{tikzpicture}
\caption{An example of the over-counting of over-counting.}\label{fig:examplepsi3}
\end{figure}

\begin{proposition}\label{prop:orb22}
    For $\nu_2(n)\geq 2$, and $\nu_2(j)\geq 1$, we have
    \[
    |\SN_2^f| = \frac{1}{2}\binom{\frac n2}{\frac j2}+ \frac12 |\SoN^f|.
    \]
\end{proposition}   
\begin{proof}
From the earlier argument in this section, we have
    \begin{align*}
        |\SN_2^f|  = & |\Orb(C_{n-2},\Neck(n-2,j-2))_1^f|+| \Orb(C_{n-2},\Neck(n-2,j))_1^f| \\
        & - \frac12\left(|\Orb_{odd}(C_{n-2},\Neck(n-2,j-2))_1^f|+| \Orb_{odd}(C_{n-2},\Neck(n-2,j))_1^f|\right.\\
        & -  \left.|\SoN^f|\right).  \\
        \end{align*}
        Using the facts $|\Orb(C_{n-2},\Neck(n-2,j-2))_1^f| = |\Orb_{even}(C_{n-2},\Neck(n-2,j-2))_1^f|+|\Orb_{odd}(C_{n-2},\Neck(n-2,j-2))_1^f|$, $|\Orb(C_{n-2},\Neck(n-2,j))_1^f| = |\Orb_{even}(C_{n-2},\Neck(n-2,j))_1^f|+|\Orb_{odd}(C_{n-2},\Neck(n-2,j))_1^f|$, and applying Proposition \ref{prop:S1}, we get
        \begin{align*}
         |\SN_2^f|  = & \frac{1}{2}\left(\binom{\frac n2-1}{\frac j2-1}+\binom{\frac n2-1}{\frac j2}+|\Orb_{odd}(C_{\frac n2-1},\Neck(\frac n2-1,\frac j2-1))_1^f|\right. \\
         & + | \Orb_{odd}(C_{\frac n2-1},\Neck(\frac n2-1,\frac j2))_1^f|-|\Orb_{odd}(C_{n-2},\Neck(n-2,j-2))_1^f|\\
         &-\left.| \Orb_{odd}(C_{n-2},\Neck(n-2,j))_1^f| + |\SoN^f| \right)\\
         = & \frac{1}{2}\binom{\frac n2}{\frac j2}+ \frac12 |\SoN^f|
    \end{align*}
    where in the second identity, we have used the Pascal identity, and the same argument as in the first sentence of the proof of Lemma \ref{lem:ooddvalue}.
    \end{proof}

    \begin{proposition}\label{prop:case4}
 For $\nu_2(n)\geq 2$ and  $\nu_2(j)\geq1$, we have $\Delta(n,j) =0$.        
    \end{proposition}
    \begin{proof}
Combining Lemma \ref{lem:intersection}, Proposition \ref{prop:S1}, and Proposition \ref{prop:orb22}, we have
\begin{align*}
    |\SeN^f| = \binom{\frac n2}{\frac j2} - |\SeN^f|.
\end{align*}
Now, apply Lemma \ref{lem:ooddvalue}, we have
\begin{enumerate}
    \item If $\nu_2(n)> \nu_2(j)$,
    \[
    |\SeN^f| = \binom{\frac n2}{\frac j2}.
    \]
    \item  If $\nu_2(n)= \nu_2(j)$,
    \[
    |\SeN^f| = \binom{\frac n2}{\frac j2}- \binom{\frac{n}{2^{\nu_2(n)+1}}-\frac12}{\frac{j}{2^{\nu_2(n)+1}}-\frac12}.
    \]
    \item If $\nu_2(n) <\nu_2(j)$,
    \[
    |\SeN^f|) = \binom{\frac n2}{\frac j2}-\binom{\frac{n}{2^{\nu_2(n)+1}}-\frac12}{\frac{j}{2^{\nu_2(n)+1}}}.
    \]
\end{enumerate}
In case (1), by Lucas's theorem, we have
\[
|\SeN^f| \equiv 0\mod 2.
\]
In case (2), by Corollary \ref{cor:corKummer1} of Kummer's theorem, 
\[
\nu_2\binom{\frac{n}{2^{\nu_2(n)+1}}-\frac12}{\frac{j}{2^{\nu_2(n)+1}}-\frac12} = \nu_2\binom{\frac{n}{2^{\nu_2(n)}}}{\frac{j}{2^{\nu_2(n)}}} = \nu_2\binom{\frac n2}{\frac j2}.
\]
Therefore, $
|\SeN^f| \equiv 0\mod 2.$ In case (3), since $\frac{j}{2^{\nu_2(n)}}$ is even, and $\frac{n}{2^{\nu_2(n)}}$ is odd, then by Corollary \ref{cor:corKummer1} and \ref{cor:corKummer2} of Kummer's theorem
\[
\nu_2\binom{\frac{n}{2^{\nu_2(n)}}-1}{\frac{j}{2^{\nu_2(n)}}} = \nu_2\binom{\frac{n}{2^{\nu_2(n)}}}{\frac{j}{2^{\nu_2(n)}}} = \nu_2\binom{\frac n2}{\frac j2}.
\]
Therefore, we  have
\[
|\SeN^f| \equiv 0\mod 2,
\]
and by Lemma \ref{lem:deltaGoeven},
\[
\Delta(n,j)=(u-1)\cdot|\SeN^f|=0.
\]
\end{proof}

\subsection{Conclusion of the proof of Theorem \ref{thm:closedformula-1}}\label{subsection:expval1}

Combining Proposition \ref{prop:case1}, Proposition \ref{prop:case2}, Proposition \ref{prop:casenu2n1nu2jgeq1}, and Proposition  \ref{prop:case4}, then invoking Lemma \ref{lem:deltaGoeven}, we have
\[
\binom{L/k}{j} = \binom{n}{j} - (1 -u)\cdot \delta(n,j),
\]
where $\delta(n,j) = \left\{
\begin{array}{cc}
     1, & \frac{j-1}{2}\prec \frac{n-2}{2},  \\
     0, & \mathrm{else}
\end{array}
\right.$ and $\prec$ is as in Definition~\ref{def:prec} in the beginning of Section~\ref{subsection:nevenjodd}.
Finally,
we can conclude the proof of Theorem \ref{thm:closedformula-1} by Lucus's theorem.

\section{Twisted case - Proof of Theorem \ref{thm:closedformula-2}}
Let $n=2j$. The strategy for the proof of Theorem~\ref{thm:closedformula-2} is to rewrite $\Orb_{even}(C_n,\Neck(n,j)^{\tau})$ in terms of the orbits for the untwisted action that we have studied in the previous sections.

\subsection{Necklace interpretation of twisted orbits}\label{sec:twistenecint}

We first define a $D_{2j} \times C_2$ action on $\Neck(2j,j)$. The $D_{2j}$ action is defined to be the same as in the untwisted case defined in Section \ref{subsection:necklace_interpretation}.
The action of $C_2 = \langle e : e^2 =1 \rangle$ commutes with the $D_{2j}$ action, whose generator $e$ is called the \textit{color exchange}. As the name suggests, $e$ exchanges the colors of all the beads. More precisely, it turns the red beads of a necklace to blue, and the blue beads to red. Define the {\em twisted} action of $C_{2j}$ on $\Neck(2j,j)$ by the map $C_{2j} \to D_{2j} \times C_2$ defined by 
 \[
 r \mapsto (r,e)
 \] and the $D_{2j} \times C_2$ action just defined. To distinguish between the twisted and untwisted actions, when we view $\Neck(2j,j)$ with its twisted $C_{2j}$-action, we will write $\Neck(2j,j)^{\tau}$. Otherwise, $\Neck(2j,j)$ has the $C_{2j}$-action from the morphism $C_{2j} \to D_{2j}$, given $r \mapsto r$ as above. We have
 \[
 {L[Q]/k \choose j} = {2j \choose j} + (u - 1) \vert \Orb_{even}(C_{2j}, \Neck(2j,j)^{\tau})\vert,
 \]
 and we will denote
 \begin{equation}\label{eqn:defdeltaprime}
 \Delta'(2j,j) : = {L[Q]/k \choose j}-{2j \choose j} = (u - 1) \vert \Orb_{even}(C_{2j}, \Neck(2j,j)^{\tau})\vert.    
 \end{equation}

\subsection{Reduction to the untwisted case}\label{subsection:reduction}

\begin{lemma} The elements of $\Orb(C_{2j},\Neck(2j,j)^{\tau})$  can be interpreted as triples
    \[\SNt \cong \{([l],[l_1],[l_2]):[l]\in\SN,\phi([l]) = ([l_1],[l_2])\}/\sim,\] where in  $([l],[l_1],[l_2])$, $[l_1]$ and $[l_2]$ are ordered. The decomposition map $\phi$ is defined in \eqref{defn:phi}. The equivalence relation is defined as $([l],[l_1],[l_2])\sim(e\cdot [l],e\cdot [l_2],e\cdot [l_1])$.
\end{lemma}
\begin{proof}
    An element in $\SNt$ is uniquely determined by an element $[l]$ in $\SN$ together with a chosen bead in $[l]$. On the other hand, two different chosen beads in $[l]$ will give rise to the same $\SNt$ orbits if and only if they both lie in the  $[l_1]$ or $[l_2]$, where $\phi([l]) = ([l_1],[l_2])$. Therefore, an element in $\SNt$ can be denoted as $([l],[l_1],[l_2])$, where the second entry $[l_1]$ means that the chosen bead lies in $[l_1]$.
    However, both $([l],[l_1],[l_2])$ and $(e\cdot [l],e\cdot [l_2],e\cdot [l_1])$ give rise to the same element in $\SNt$. By quotienting out this equivalence relation, we arrive at the desired isomorphism.
\end{proof}

\begin{figure}
\centering
\begin{tikzpicture}[scale=1.0]
   \begin{scope}[shift={(0, 0)}]

    \def\radius{2cm}

    \def\numBeads{12}

    \foreach \i in {1,...,\numBeads} {
        \pgfmathsetmacro{\angle}{360/\numBeads * (\i - 1)}
        
        \ifnum\i=1
            \def\fillColor{blue}
        \else
        \ifnum\i=3
            \def\fillColor{blue}
        \else
        \ifnum\i=5
            \def\fillColor{blue}
        \else
        \ifnum\i=6
            \def\fillColor{blue}
        \else
        \ifnum\i=9
            \def\fillColor{blue}
        \else
        \ifnum\i=10
            \def\fillColor{blue}
        \else
            \def\fillColor{red}
        \fi\fi\fi\fi\fi\fi

        \ifodd\i
            \node[draw, fill=\fillColor, minimum size=12pt] at (\angle:\radius) {};
        \else

            \node[draw, fill=\fillColor, circle, minimum size=12pt] at (\angle:\radius) {};
        \fi
    }
    \end{scope}
\begin{scope}[shift={(7, 0)}]

    \def\radius{2cm}

    \def\numBeads{12}

    \foreach \i in {1,...,\numBeads} {

        \pgfmathsetmacro{\angle}{360/\numBeads * (\i -2)}
        
        \ifnum\i=1
            \def\fillColor{red}
        \else
        \ifnum\i=3
            \def\fillColor{red}
        \else
        \ifnum\i=5
            \def\fillColor{red}
        \else
        \ifnum\i=6
            \def\fillColor{red}
        \else
        \ifnum\i=9
            \def\fillColor{red}
        \else
        \ifnum\i=10
            \def\fillColor{red}
        \else
            \def\fillColor{blue}
        \fi\fi\fi\fi\fi\fi

        \ifodd\i

            \node[draw,  fill=\fillColor, minimum size=12pt] at (\angle:\radius) {};
        \else

            \node[draw, circle, fill=\fillColor,  minimum size=12pt] at (\angle:\radius) {};
        \fi
    }
    \end{scope}
\node at (3.5, 0) { $\sim$};
\end{tikzpicture}
\caption{An example of elements in $\Orb(C_{12},\Neck(12,6)^\tau)$.}\label{fig:untwistedorb}
\end{figure}
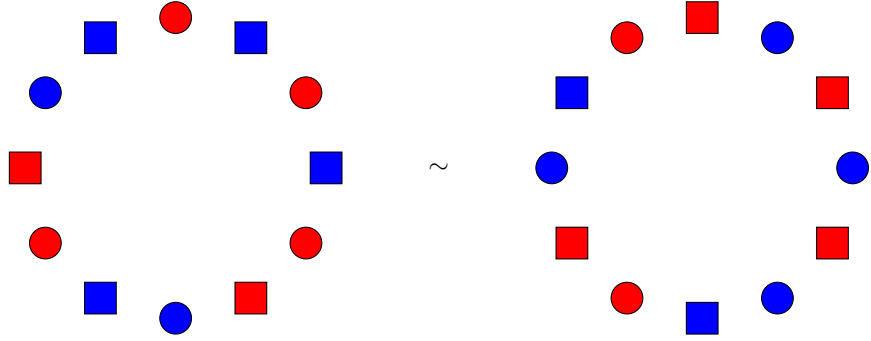

\begin{figure}
\centering
\begin{tikzpicture}[scale=1.0]
   \begin{scope}[shift={(0, 0)}]

    \def\radius{2cm}

    \def\numBeads{12}

    \foreach \i in {1,...,\numBeads} {

        \pgfmathsetmacro{\angle}{360/\numBeads * (\i - 1)}
        
        \ifnum\i=1
            \def\fillColor{blue}
        \else
        \ifnum\i=3
            \def\fillColor{blue}
        \else
        \ifnum\i=5
            \def\fillColor{blue}
        \else
        \ifnum\i=6
            \def\fillColor{blue}
        \else
        \ifnum\i=9
            \def\fillColor{blue}
        \else
        \ifnum\i=10
            \def\fillColor{blue}
        \else
            \def\fillColor{red}
        \fi\fi\fi\fi\fi\fi

        \ifodd\i

            \node[draw, fill=\fillColor, minimum size=12pt] at (\angle:\radius) {};
        \else

            \node[draw, fill=\fillColor, circle, minimum size=12pt] at (\angle:\radius) {};
        \fi
    }
    \end{scope}
\begin{scope}[shift={(7, 0)}]

    \def\radius{2cm}

    \def\numBeads{12}

    \foreach \i in {1,...,\numBeads} {

        \pgfmathsetmacro{\angle}{360/\numBeads * (\i - 1)}
        
        \ifnum\i=1
            \def\fillColor{blue}
        \else
        \ifnum\i=3
            \def\fillColor{blue}
        \else
        \ifnum\i=5
            \def\fillColor{blue}
        \else
        \ifnum\i=6
            \def\fillColor{blue}
        \else
        \ifnum\i=9
            \def\fillColor{blue}
        \else
        \ifnum\i=10
            \def\fillColor{blue}
        \else
            \def\fillColor{red}
        \fi\fi\fi\fi\fi\fi

        \ifodd\i

            \node[draw, circle, fill=\fillColor, minimum size=12pt] at (\angle:\radius) {};
        \else

            \node[draw, fill=\fillColor,  minimum size=12pt] at (\angle:\radius) {};
        \fi
    }
    \end{scope}
\node at (3.5, 0) { $\xrightarrow{s}$};
\end{tikzpicture}
\caption{An example of swapping.}\label{fig:swapping}
\end{figure}
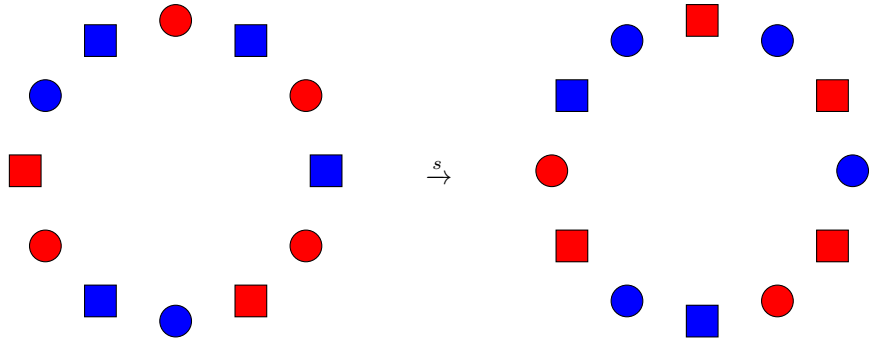
One way to visualize the triple is shown in Figure \ref{fig:untwistedorb}. Denote the element as $([l],[l_1],[l_2])\in\Orb(C_{12},\Neck(12,6)^\tau)$.  The  circle-shaped beads form $[l_1]\in \Orb(C_6,\Neck(6,2))$, whereas the square-shaped beads form $[l_2]\in \Orb(C_6,\Neck(6,4))$.

From now on, we will use the triple $([l],[l_1],[l_2])$ to denote elements in $\SNt$.
Define the twisted period as $\pi'([l],[l_1],[l_2]) := |([l],[l_1],[l_2])|$, i.e., the size of the twisted $C_{2j}$-orbit. Next, we have a $C_2$ action on $\SNt$, which acts by exchanging $l_1,l_2$. Combinatorically, this action interchanges the two potentially different twisted orbits that could come from the same underlining untwisted orbit $[l]$.  We will call this action swapping, and an example is shown in Figure \ref{fig:swapping}. For a set $\cN$, we will use $\cN^s$ to denote the subset fixed by swapping. 

\begin{lemma}\label{lem:deltaprimevalue}
    $\Delta'(2j,j) = (u-1)|\SNt^s|$, where $\SNt^s$ is the subset of $\SeNt$ that is fixed by swapping.
\end{lemma}
\begin{proof}
    Since the swapping action does not change the twisted period, by \eqref{eqn:defdeltaprime} and $2(u-1)=0$, we got the desired formula.
\end{proof}

\begin{lemma}\label{lem:twistedfixed}
In $\Orb(C_{2j},\Neck(2j,j)^{\tau})$, 
    $([l],[l_1],[l_2])=      ([l],[l_2],[l_1])$ if and only if either $[l_1]=[l_2]$ or $[l] = e\cdot [l]$.
\end{lemma}
\begin{proof}
    By the definition of the twisted orbit, there are two possibilities for $([l],[l_1],[l_2])=      ([l],[l_2],[l_1])$. The first is $[l_1] = [l_2]$, and the other is $[l] = e\cdot[l], [l_1] = e\cdot[l_1],[l_2]=e\cdot[l_2]$.
\end{proof}

\begin{corollary}\label{cor:deltatwistedodd}
    $\Delta'(2j,j) =0$, for $\nu_2(j)=0$.
\end{corollary}
\begin{proof}
    Since $j$ is odd, $[l_1]\neq [l_2]$ and $[l_1]\neq e\cdot[l_2]$ for all elements $([l],[l_1],[l_2])\in \SNt$. Thus, the $C_2$-action of swapping acts freely, so $|\SNt^s|\equiv0\mod2$.
\end{proof}

From now on, we will assume $j$ is even. 
Applying the above results to $|\SNt^s|$, we have
\begin{align*}
|\SNt^s|   =    & |([l],[l_1],[l_1])\in \SeNt, [l]\neq e\cdot[l]  | \\
& + |([l],[l_1],[l_2])\in \SeNt,[l]=e\cdot[l]  |. \\
\end{align*}
For $([l],[l_1],[l_1])\in \SeNt$, $[l]\neq e\cdot [l]\Leftrightarrow [l_1]\neq e\cdot [l_1]$, and  notice that if $[l_1] \neq e\cdot[l_1]$, then $2|\pi'([l],[l_1],[l_1])$. Therefore,
\begin{align*}
|\SNt^s| 
= & \sum_{[l]\in \Orb(C_{j},\Neck(j,\frac j2)),[l]\neq e\cdot[l]}\frac{|\phi^{-1}([l],[l])|}{2} \\
&+  |([l],[l_1],[l_2])\in \SeNt,[l]=e\cdot[l]|
\end{align*}

To further simplify the second term, we notice the following facts.
\begin{lemma}
   For $([l],[l_1],[l_2])\in \SNt$, we have $\pi([l]) \equiv 0\mod 2$.
\end{lemma}
\begin{proof}
    Since $\nu_2(2j)>\nu_2(j)$, assume $2\nmid\pi([l])$, then the number of blue beads will not be an integer.
\end{proof}
\begin{lemma}
    If $[l] = e\cdot [l]$, then for all $l\in \Neck(2j,j)$, we have
    \[
    e\cdot l = r^{\frac{\pi([l])}{2}}\cdot l.
    \]
\end{lemma}
\begin{proof}
    Apparently $e= r^{m}$ for some $m\in\bZ/\pi([l])\bZ$. Since $e$ changes $l$, we can take $0<m<\pi([l])$. On the other hand, since $e^2 = 1$, one has $\pi([l])|2m$, thus $m = \frac{\pi([l])}{2}$.
\end{proof}

\begin{proposition}
     $\pi'([l],[l_1],[l_2])=\pi([l])$ unless $l = e\cdot[l]$, and $\nu_2(\pi([l])) =1$, in which case $\pi'([l],[l_1],[l_2])=\frac{\pi([l])}{2}\equiv 1\mod 2$.
\end{proposition}
\begin{proof}
    Since $\pi([l])$ is even, $r^{\pi([l])} = e^{\pi([l])}$, and we always have $\pi'([l],[l_1],[l_2])\leq\pi([l])$. By the definition of the twisted action, if $\pi'([l],[l_1],[l_2])$ is even, we have  $\pi'([l],[l_1],[l_2])=\pi([l])$. Therefore, the only possibility for $\pi'([l],[l_1],[l_2])<\pi([l])$ is when $\pi'([l],[l_1],[l_2])$ is odd. In this case, we have $2\pi'([l],[l_1],[l_2]) = \pi([l])$, and thus $\nu_2([l])=1$.
\end{proof}

Therefore
\begin{align*}
|\SNt^s| = &\sum_{[l]\in \Orb(C_{j},\Neck(j,\frac j2)),[l]\neq e\cdot[l]}\frac{|\phi^{-1}([l],[l])|}{2} \\ & + |([l],[l_1],[l_2])\in \SNt,[l]=e\cdot[l],\nu_2(\pi([l]))>1|\\ 
= & \sum_{[l]\in \Orb(C_{j},\Neck(j,\frac j2)),[l]\neq e\cdot[l]}\frac{|\phi^{-1}([l],[l])|}{2} \\
& + |l\in \SNjj,[l]=e\cdot[l],\nu_2(\pi([l]))>1|
\end{align*}
where the last formula is an enumeration purely in terms of untwisted orbits.

\subsection{Relations to partitions}\label{subsection:relatetopart}

Consider an element $[l]\in\SN$. One can uniquely write $[l]$ as the cyclic orbit of a partition of $n$ as follows. Call a maximal consecutive segment of beads of $l$ of the same color a \textit{cluster}. By replacing each cluster of $l$ by its number of beads, one obtains the corresponding partition on $n$. For example, fix $l\in [l]$, if the beads in positions 1 and 5 are red but everything in between is blue, we will replace the 3 blue beads with the number 3. If the cluster is red, we will mark it by adding an underline to distinguish it from the blue ones. We will use the notation $(\cdots)$ to denote the cyclic equivalent class of partitions. So $[l]$ can be uniquely written as the cyclic orbit of a partition of $n$ as
    \[
    [l] = (\underline{r_1}b_1\underline{r_2}b_2\cdots \underline{r_m}b_m),\quad r_i,b_i\in\bZ_{>0},
    \]
    with $\sum_{i}(r_i+b_i)=n,\sum_{i}b_i = j.$ Under this notation, the action of $e$ is simply 
    \[
    e\cdot (\underline{r_1}b_1\underline{r_2}b_2\cdots \underline{r_m}b_m) = (r_1\underline{b_1}r_2\underline{b_2}\cdots r_m \underline{b_m}).
    \]
    Apparently, the length of the partition has to be even.   For example, the three elements in Figure \ref{fig:orb104} can be written as $(\underline{6} 4),(\underline{2}1\underline{1}1\underline{2}1\underline{1}1)$, $(\underline{1}1\underline{4}1\underline{1}2)$, respectively.
    
    Denote the set of even length partitions of $n$ with an underlined marking for every other entry such that the sum of unmarked entries is $j$ by $\mathrm{Part}(n,j,\underline{n-j})$. We further denote the set of its cyclic equivalent classes by $\Orb(C,\mathrm{Part}(n,j,\underline{n-j}))$. For an element $[p]\in \Orb(C,\mathrm{Part}(n,j,\underline{n-j}))$, we denote its length by $|p| := \mathrm{length}(p)$, and define its period as $\varpi([p]) = |C_{|p|}\cdot p|$\footnote{Note there is no direct relation between $\varpi$ and $\pi'$.}.

    Denote the set of partitions of $j$ by $\mathrm{Part}(j)$, and similarly denote the set of cyclic equivalent classes and orbits as in the marked case, then we have
\begin{lemma}\label{lem:neckorbtopart}
    There is a bijection
    \[
     \{[l]\in \SNjj,l=e\cdot[l]\}  \cong \{[p]\in\Orb(C,\mathrm{Part}(j)),\nu_2(\varpi([p]))=0\}.
    \]
\end{lemma}
\begin{proof}
Let $[l]\in \SNjj$, and $[p]= (\underline{r_1}b_1\underline{r_2}b_2\cdots \underline{r_m}b_m)$ be its corresponding element in $\Orb(C,\mathrm{Part}(2j,j,\underline{j}))$. If $[l]=e\cdot[l]$, then
\[
(\underline{r_1}b_1\underline{r_2}b_2\cdots \underline{r_m}b_m) = (r_1\underline{b_1}r_2\underline{b_2}\cdots r_m \underline{b_m}),
\]
and apparently we need $(r_1\cdots r_m) = (b_1\cdots b_m)$. Moreover, by passing to the smaller partition that is formed by taking $\varpi([p])$ consecutive parts, without loss of generality, we can assume the partition is non-periodic, i.e., $\varpi([p]) = 2m$. Then the only element in $\Orb(C,\mathrm{Part}(2j,j,\underline{j}))$ of the form $(\underline{r_1},r_{i_1},\underline{r_2},r_{i_2},\cdots,\underline{r_{m}},r_{i_m})$ that satisfies $[p] = e\cdot [p]$ is
\[
(\underline{r_1}r_{\frac{m+3}{2}}\underline{r_2}r_{\frac{m+5}{2}}\cdots r_{m}\underline{r_{\frac{m+1}2}}r_1\underline{r_{\frac{m+3}{2}}}\cdots\underline{r_{m}}r_{\frac{m+1}{2}}),
\]
and $m$ has to be odd.
\end{proof}

Recall that the total number of partitions of $j\in\bZ_{>0}$ is $2^{j-1}$. So
\begin{equation}\label{eqn:decomppart}
2^{j-1} = \sum_{d}d |\{[p]\in\Orb(C,\mathrm{Part}(j)),\varpi([p])=d\}|,
\end{equation}

\begin{lemma}
For $j>1$, we have
    $|\{l\in \SNjj,[l]=e\cdot[l],\nu_2(\pi([l]))>1\}| \equiv |\{l\in \SNjj,[l]=e\cdot[l],\nu_2(\pi([l]))=1\}|\mod2$.
\end{lemma}
\begin{proof}
By \eqref{eqn:decomppart}, we have
\begin{align*}
    |\{[p]\in\Orb(C,\mathrm{Part}(j)),\nu_2(\varpi([p]))=0\}| = &\sum_{\nu_2(d)=0}d|\{[p]\in\Orb(C,\mathrm{Part}(j)),\varpi([p])=d\}|\\
    \equiv & \sum_{d}d|\{[p]\in\Orb(C,\mathrm{Part}(j)),\varpi([p])=d\}|\mod2\\
    = & 2^{j-1}\equiv 0\mod2.
\end{align*}
Then by Lemma \ref{lem:neckorbtopart}, we have
\[
|\{l\in \SNjj,[l]=e\cdot[l]\}|\equiv 0\mod2.
\]
\end{proof}

Therefore, we have
\begin{equation}\label{eqn:secondtothelast}
    \begin{split}
        |\SNt^s| 
\equiv & \sum_{[l]\in \Orb(C_{j},\Neck(j,\frac j2)),[l]\neq e\cdot[l]}\frac{|\phi^{-1}([l],[l])|}{2} \\
& + |\{l\in \SNjj,[l]=e\cdot[l],\nu_2(\pi([l]))=1\}|\mod2.
    \end{split}
\end{equation}

\subsection{Conclusion of the proof of Theorem \ref{thm:closedformula-2}}\label{subsection:conc2}
First, we have:
\begin{lemma}
For $[l]\in \Orb(C_{j},\Neck(j,\frac j2))$, we have
$|\phi^{-1}([l],[l])| = \frac{\pi([l])}{2}$.
\end{lemma}
\begin{proof}
    Similar to Lemma \ref{lem:fibperiod1}.
\end{proof}
Then
\begin{align*}
\sum_{[l]\in \Orb(C_{j},\Neck(j,\frac j2)),[l]\neq e\cdot[l]}\frac{|\phi^{-1}([l],[l])|}{2}  = &\sum_{[l]\in \Orb(C_{j},\Neck(j,\frac j2)),[l]\neq e\cdot[l]}\frac{|\pi([l])|}{2}\\
 = & \sum_{[l]\in \Orb(C_{j},\Neck(j,\frac j2))}\frac{|\pi([l])|}{2} - \sum_{[l]\in \Orb(C_{j},\Neck(j,\frac j2)),[l]= e\cdot[l]}\frac{|\pi([l])|}{2}\\
\end{align*}
Since $\pi([l])\equiv0\mod2$, we have
\[
\sum_{[l]\in \Orb(C_{j},\Neck(j,\frac j2)),[l]= e\cdot[l]}\frac{|\pi([l])|}{2}\equiv |[l]\in \Orb(C_{j},\Neck(j,\frac j2)),[l]= e\cdot[l],\nu_2(\pi(l))=1|\mod 2.
\]
As 
\[
\{[l]\in \Orb(C_{j},\Neck(j,\frac j2)),[l]= e\cdot[l],\nu_2(\pi(l))=1\} =\{[l]\in \Orb(C_{j},\Neck(2j, j)),[l]= e\cdot[l],\nu_2(\pi(l))=1\},
\]
we can conclude
\begin{proposition}\label{prop:deltatwistedjeven}
    $\Delta'(2j,j) = \frac12\binom{2j}{j}\cdot (u-1) $, for $j\equiv 0\mod2$.
\end{proposition}
\begin{proof}
    By \eqref{eqn:secondtothelast}, and the calculation in this subsection, we have
    \[
    |\SNt^s| 
\equiv \frac12\binom{j}{\frac j2}\mod 2,
    \]
    Then by Lemma \ref{lem:deltaprimevalue} and Corollary \ref{cor:corKummer1} of Kummer's theorem, we have
    \[
    \Delta'(2j,j) = \frac12\binom{2j}{j}\cdot (u-1) 
    \]
\end{proof}
The proof of Theorem \ref{thm:closedformula-2} is concluded by combining Proposition \ref{prop:deltatwistedjeven} and Corollary \ref{cor:deltatwistedodd}.

\subsection{Explicit value of twisted binomial coefficients}\label{subsection:expval}
By Corollary \ref{cor:corKummer3} of Kummer's theorem, we have
\[
{L[Q]/k \choose j} = {2j \choose j}+(u-1)\delta(j),
\]
where $\delta(j)=\left\{\begin{array}{cc}
    1 &  j = 2^m,m\in\bZ_{>0} \\
   0  & \textrm{else}
\end{array}\right.$.

See Example~\ref{ex:twisted-enriched-bin-finite-field} for computations with small $j$.

\bibliographystyle{amsplain}
\bibliography{Bibli}

\end{document}